\documentclass[12pt,letterpaper]{article}

\usepackage[margin=1in]{geometry}

\usepackage{amsthm}
\usepackage{ucs}
\usepackage[utf8x]{inputenc}
\usepackage[english]{babel}
\usepackage{amsmath}
\usepackage{amsfonts}
\usepackage{bbm}
\usepackage{bm}
\usepackage{amssymb}
\usepackage{dsfont}

\usepackage{mathtools}
\usepackage{mathrsfs}
\usepackage{graphicx}
\usepackage{algorithmic}
\usepackage{algorithm}
\usepackage{booktabs}
\usepackage{multirow}
\usepackage{fancyhdr}
\usepackage{enumerate}
\usepackage{setspace}
\usepackage{color}
\usepackage{url}

\usepackage[table]{xcolor}
\usepackage{picture}

\usepackage{natbib}

\usepackage{soul}

\theoremstyle{plain}\newtheorem{theorem}{Theorem}[section]
\theoremstyle{plain}\newtheorem{corollary}[theorem]{Corollary}
\theoremstyle{definition}
\theoremstyle{plain}
\theoremstyle{plain}\newtheorem{lemma}[theorem]{Lemma}
\theoremstyle{plain}\newtheorem{proposition}[theorem]{Proposition}
\theoremstyle{definition}

\newcommand{\N}{{\mathds N}}
\newcommand{\R}{{\mathds R}}

\usepackage{graphicx}

\usepackage{mathptmx}

\usepackage{graphicx}
\usepackage[T1]{fontenc}
\usepackage{ucs}
\usepackage{amsmath}
\usepackage{amsfonts}
\usepackage{bbm}
\usepackage{bm}
\usepackage{amssymb}
\usepackage{mathtools}
\usepackage{mathrsfs}
\usepackage[mathcal]{euscript}

\usepackage{algorithm}
\numberwithin{algorithm}{section}
\usepackage{algorithmic}

\usepackage{array}
\usepackage{multirow}
\usepackage{enumerate}

\usepackage{algorithm}
\numberwithin{algorithm}{section}
\usepackage{algorithmic}
\usepackage{booktabs}
\usepackage{enumerate}
\usepackage{url}

\usepackage[table]{xcolor}

\newcommand{\ud}{\mbox{d}}
\newcommand{\eps}{\varepsilon}

\DeclareMathOperator{\var}{\mathsf{Var}}

\DeclareMathOperator{\E}{\mathsf{E}}
\DeclareMathOperator{\ind}{\mathbbm{1}}

\newcommand{\prob}{\mathsf{P}}

\newcolumntype{x}[1]{%
	>{\raggedleft\hspace{0pt}}p{#1}}%

\DeclareMathAlphabet{\mathitbf}{OML}{cmm}{b}{it}

\definecolor{seda}{gray}{0.9}

\makeatletter 
\newcommand{\vast}{\bBigg@{4}}
\newcommand{\Vast}{\bBigg@{5}}
\makeatother

\floatname{algorithm}{Procedure}

\newenvironment{proof*}[1]
{\renewcommand{\proofname}{#1}%
	\begin{proof}}
	{\end{proof}}

\renewcommand{\P}{\mathsf{P}} 

\renewcommand{\d}{\ensuremath{\delta}}

\newcommand{\e}{\varepsilon} 

\newcommand{\dist}{\mathsf{D}} 
\newcommand{\weak}{\mathsf{D}[0,1]}

\newcommand{\Op}{\ensuremath{O_{\prob}}} 
\newcommand{\op}{\ensuremath{o_{\prob}}} 

\DeclareMathOperator*{\plim}{plim}

\newtheorem{assumpC}{Assumption}

\newtheorem{assumpE}{Assumption}

\newtheorem{assumpV}{Assumption}

\allowdisplaybreaks

\newcommand{\blind}{1}

\begin{document}

\def\spacingset#1{\renewcommand{\baselinestretch}%
{#1}\small\normalsize} \spacingset{1}


\if1\blind
{
  \title{\bf Nuisance Parameters Free Changepoint Detection in Non-stationary Series}
  \author{
  	Michal Pe\v{s}ta\thanks{
  	This research was supported by the Czech Science Foundation project GA\v{C}R No.~18-01781Y.}\hspace{.2cm}\\
    Department of Probability and Mathematical Statistics\\Charles University, Prague, Czech Republic\\
    and \\
    Martin Wendler
    \hspace{.2cm}\\
    Institute of Mathematics and Computer Science\\University of Greifswald, Germany}
  \maketitle
} \fi

\if0\blind
{
  \bigskip
  \bigskip
  \bigskip
  \begin{center}
    {\LARGE\bf Title}
\end{center}
  \medskip
} \fi

\bigskip
\begin{abstract}
Detecting abrupt changes in the mean of a~time series, so-called changepoints, is important for many applications. However, many procedures rely on the estimation of nuisance parameters (like long-run variance). Under the alternative (a change in mean), estimators might be biased and data-adaptive rules for the choice of tuning parameters might not work as expected. If the data is not stationary, but heteroscedastic, this becomes more challenging. The aim of this paper is to present and investigate two changepoint tests, which involve neither nuisance nor tuning parameters. This is achieved by combing self-normalization and wild bootstrap. We study the asymptotic behavior and show the consistency of the bootstrap under the hypothesis as well as under the alternative, assuming mild conditions on the weak dependence of the time series and allowing the variance to change over time. As a~by-product of the proposed tests, a~changepoint estimator is introduced and its consistency is proved. The results are illustrated through a~simulation study, which demonstrates computational efficiency of the developed methods. The new tests will also be applied to real data examples from finance and hydrology.
\end{abstract}

\noindent%
\vfill

\newpage
\spacingset{1.45} 

\section{Main goals}

In the statistical analysis, it is of particular interest to be able to detect systematic changes---so-called \emph{changepoints}---in the underlying structure despite the random fluctuations and to estimate the time of these changes. Under the assumption of finite expectations, changes in the location are typically detected by comparing sample means and the asymptotic distribution can be derived from an~invariance principle for the partial sum process.

However, one has to estimate the \emph{long-run variance} to utilize the traditional CUSUM-statistic and this involves some difficulties. For time series, the long-run variance includes the covariances, which have to be estimated and combined with, e.g., kernels. Under the alternative, the estimation of the covariances is biased, so the long-run covariance is typically overestimated, which results in a~loss of power, see~\cite{huvskova2010note}. In many applications, the observations do not seem to be stationary even under the hypothesis of no change in mean, because the amount of fluctuation is not constant. To estimate the time-varying long-run variance is even more difficult. In a~recent article, \cite{gorecki2017change} have followed this approach.

The main aim of this paper is to develop tests for the hypothesis of a~constant expectation against the alternative of at most one changepoint that avoid the problems of long-run variance estimation and heteroscedasticity. Our new test statistics will not involve \emph{any nuisance parameters} and will work for \emph{heteroscedastic and dependent} time series under some mild mixing conditions. Additionally, we will give a~consistent estimator for the time of the change.

Some authors proposed to use nonparametric resampling methods like bootstrap (e.g., \cite{HK2012} and~\cite{PP2018}) or subsampling (e.g., \cite{betken2017subsampling}) to avoid the estimation of the long-run variance. However, these methods still involve the choice of tuning parameters like bandwidths or block sizes and only work for stationary time series. Other approaches are \emph{ratio statistics} and \emph{self-normalized statistics}, which do not rely on tuning parameters. Ratio tests have been introduced to detect changes in persistence by \cite{Kim2000} and since have been studied for changes in mean~\citep{HHH2008}, for changes in variance~\citep{zhao2011ratio}, for heavy-tailed sequences~\citep{dan2017detection}, for panel date framework~\citep{PP2015}, and for robust $M$-estimators~\citep{PP2018}.

Self-normalized test statistics for changepoints were firstly proposed by \cite{shao2010testing} and were generalized to long range dependent time series~\citep{Shao2011}. \cite{betken2016} developed a~robust self-normalized test based on the Wil\-coxon-statistic, \cite{zhang2018unsupervised} proposed a~self-normalized test for multiple changepoints. Our approach is to combine new variants of the self-normalized test statistics with the \emph{wild bootstrap}. The wild bootstrap was proposed by \cite{wu1986jackknife} and is consistent under heteroscedasticity. However, it does not reproduce the dependence of the data. We will show that under our model assumptions, it still gives the correct critical values for the self-normalized test statistics. In this way, we can avoid using the dependent wild bootstrap of \cite{shao2010dependent}, which involves the choice of a~kernel and of a~bandwidth parameter.

The paper is organized as follows: In the next section, we will introduce our data model and our new test statistics. In Section~\ref{sec:mainresult}, the technical assumptions are discussed, the bootstrap is introduced and our main theoretical results are presented. We provide a~table with critical values of our test statistics under stationarity in Section~\ref{sec:simu} and investigate the finite sample properties through simulation results. Two data examples from finance and hydrology are provided in Section~\ref{sec:data}. Afterwards, our conclusion follows. The proofs of our theoretical results can be found in the appendix.

\section{Stochastic model and methods}\label{sec:model}

\subsection{Changepoint model}

We tend to study time series with one abrupt change in the mean at an unknown point in time. Let us consider observations $Y_{1,n},\ldots,Y_{n,n}$ obtained at time ordered points. We are interested in testing the null hypothesis of all observations being random variables having equal expectation. Our goal is to test against the alternative of the first $\tau_n$ observations have expectation $\mu$ and the remaining $n-\tau_n$ observations come from distributions with expectation $\mu + \d_n$, where $\d_n\neq 0$. More precisely, our model is
\begin{equation}\label{eq:locmodel}
Y_{n,k}=\mu+\d_n\ind\{k>\tau_n\} + \sigma\left(\frac{k}{n}\right)\eps_k,\quad k=1,\ldots,n,
\end{equation}
where $\mu$, $\d_n$, and $\tau_n$ are unknown parameters, $\{Y_{n,k}\}_{n=1,k=1}^{\infty,n}$ is a~\emph{triangular array} of random variables, $\{\eps_{n}\}_{n=1}^{\infty}$ is a~sequence of stationary centered disturbances, $\sigma(t)$ is a~non-stochastic variance function, and $\ind\{A\}$ denotes the indicator of set~$A$. The time point~$\tau_n$ is called the \emph{changepoint}. A~similar model was assumed by \cite{gorecki2017change}.


We are going to test the null hypothesis that no change occurred against the alternative that a~change occurred at some unknown time point $\tau_n$, i.e.,
\begin{equation*}
\mathcal{H}_0: \tau_n=n\quad\mbox{versus}\quad\mathcal{H}_1:\tau_n<n,\ \d_n \ne 0.
\end{equation*}

\subsection{Test statistics}

The \emph{CUSUM-statistic} is frequently used to detect changes in the mean and it is based on the partial sums $\sum_{i=1}^{k}\left(Y_{n,i}-\bar{Y}_{n,1:n}\right)$, $k=1,\ldots,n-1$ of the centered observations, where
\[
\bar{Y}_{n,i:j}=\frac{1}{j-i+1}\sum_{k=i}^{j}Y_{n,k},\quad i\leq j.
\]
To combine the values of the partial sums for different~$k$ into a~single test statistic, one can use the \emph{supremum-type} CUSUM-statistic $\max_{k=1,\ldots,n-1}\left|\sum_{i=1}^{k}\big(Y_{n,i}-\bar{Y}_{n,1:n}\big)\right|$ or the \emph{integral-type} CUSUM-statistic $\sum_{k=1}^{n-1}\left(\sum_{i=1}^{k}\left(Y_{n,i}-\bar{Y}_{n,1:n}\right)\right)^2$. These test statistics need to be standardized by a~variance of the series. However, it is practically \emph{difficult to find a~variance estimator} with satisfactory properties. Such difficulty can occur in situations with dependent or heteroscedastic random errors. Nonetheless, the variance estimators often do not perform well even in the i.i.d.~case, especially under alternatives~\citep{Antoch1997}.

To avoid the estimation of variance parameters, different ratios of such test statistics have been proposed. \cite{HHH2008} divide the supremum-test statistic of the first part of the series by the supremum-test statistic of the second part of the data. \cite{wenhua2016ratio} use a~ratio test based on the integral-type statistic. \cite{shao2010testing} introduced a self-normalized statistic, which uses the supremum-type CUSUM-statistic of the whole data set in the numerator, divided by the sum of two integral-type statistics of the data before~$k$ and after~$k$.

Our idea is to use a~self-normalization of the CUSUM-statistic by the same type: We divide the supremum-type statistic by two supremum-type statistics and the integral-type statistic by two integral-type statistics. Our test statistics can be expressed as functionals of the \emph{cumulative sums}
\[
V_n(k):=\sum_{i=1}^{k}Y_{n,i}\quad\mbox{and}\quad \widetilde{V}_n(k):=V_n(n)-V_n(k).
\]
We define the \emph{self-normalized test statistics} as
\begin{align}\label{eq:statistic1}
\mathscr{Q}(V_n)&:=\max_{1\leq k \leq n}\Bigg|\frac{V_n(k)-k/n V_n(n)}{\max\limits_{1\leq i \leq k}\big|V_n(i)-i/kV_n(k)\big|+\max\limits_{k< i \leq n}\big|\widetilde{V}_n(i)-(n-i)/(n-k)\widetilde{V}_n(k)\big|}\Bigg|\\
&\equiv \max_{1\leq k \leq n}\frac{\left|\sum_{i=1}^{k}\left(Y_{n,i}-\bar{Y}_{n,1:n}\right)\right|}{\max\limits_{1\leq i \leq k}\left|\sum_{j=1}^{i}\left(Y_{n,j}-\bar{Y}_{n,1:k}\right)\right|+\max\limits_{k< i \leq n}\left|\sum_{j=i}^{n}\left(Y_{n,j}-\bar{Y}_{n,(k+1):n}\right)\right|}\notag
\end{align}
and
\begin{align}\label{eq:statistic2}
\mathscr{R}(V_n)&:=\sum_{k=1}^n\frac{\left\{V_n(k)-k/n V_n(n)\right\}^2}{\sum_{i=1}^k\big\{V_n(i)-i/kV_n(k)\big\}^2+\sum_{i=k+1}^n\big\{\widetilde{V}_n(i)-(n-i)/(n-k)\widetilde{V}_n(k)\big\}^2}\\
&\equiv \sum_{k=1}^n\frac{\left\{\sum_{i=1}^{k}\left(Y_{n,i}-\bar{Y}_{n,1:n}\right)\right\}^2}{\sum_{i=1}^k\left\{\sum_{j=1}^{i}\left(Y_{n,j}-\bar{Y}_{n,1:k}\right)\right\}^2+\sum_{i=k+1}^n\left\{\sum_{j=i}^{n}\left(Y_{n,j}-\bar{Y}_{n,(k+1):n}\right)\right\}^2}.\notag
\end{align}

For many changepoint tests, one has to skip, for instance, the first and the last $10\%$ of observations as possible candidates for a~changepoint, see, e.g., \cite{shao2010testing}. Moreover, the amount of trimming can be viewed as an~additional tuning parameter. For our test statistics, we are able to consider all time points $k=1,\ldots,n$. The limit distribution of our statistics is obtained with the help of the continuous mapping theorem, using limit theorems for the partial sum process under weak dependence and heteroscedasticity by~\cite{Cav2005}.

\section{Main results}\label{sec:mainresult}

\subsection{Assumptions}
Prior to deriving asymptotic properties of the test statistic, we summarize the notion of \emph{strong mixing} ($\alpha$-mixing) dependence in more detail, which will be imposed on the model's errors. Suppose that $\{\e_n\}_{n=1}^{\infty}$ is a~sequence of random elements on a~probability space $(\Omega,\mathcal{F},\P)$. For sub-$\sigma$-fields $\mathcal{A},\mathcal{B}\subseteq\mathcal{F}$, we define $\alpha(\mathcal{A},\mathcal{B}):=\sup_{A\in\mathcal{A},B\in\mathcal{B}}\left|\P(A\cap B)-\P(A)\P(B)\right|$. Intuitively, $\alpha(\cdot,\cdot)$ measures the dependence of the events in $\mathcal{B}$ on those in $\mathcal{A}$. 
There are many ways in which one can to describe weak dependence or, in other words, \emph{asymptotic independence} of random variables, see~\cite{Bradley2005}. Considering a~filtration $\mathcal{F}_m^n:=\sigma\{\e_i\in\mathcal{F},m\leq i\leq n\}$, sequence $\{\e_n\}_{n=1}^{\infty}$ of random variables is said to be \emph{strong mixing} ($\alpha$-mixing) if $\alpha(n):=\sup_{k\in\mathbbm{N}}\alpha(\mathcal{F}_{1}^k,\mathcal{F}_{k+n}^{\infty})\to 0$ as $n\to\infty$.

We proceed to the assumptions that are needed for deriving asymptotic properties of the proposed test statistics. For the functional central limit theorem, we need assumptions controlling the dependence and the moments of the underlying errors.
\begin{assumpE}\label{assump:UP1}
	$\left\{\eps_{n}\right\}_{n=1}^{\infty}$ form a~zero-mean strictly stationary $\alpha$-mixing sequence such that $\var\eps_{n}=1$, $\E|\eps_{n}|^p<\infty$ for some $p>2$ with mixing coefficients $\alpha(n)$ satisfying $\sum_{n=1}^{\infty}\{\alpha(n)\}^{2(1/r-1/p)}<\infty$ for some $r\in(2,4]$, $r\leq p$, and, additionally, $\sum_{n=1}^{\infty}n\alpha(n)<\infty$. Furthermore, for the long-run variance, it holds that $0<\lambda:=1+2\sum_{n=1}^{\infty}\E\eps_1\eps_{n+1}<\infty$.
\end{assumpE}

The mixing properties will be inherited by $Y_{n,k}=\mu+\delta_n\ind\{j>\tau_n\}+\sigma(k/n)\epsilon_k$, but this process can additionally model heteroscedasticity, which is important for many applications. To control variability of the series, we have an~assumption regarding heteroscedasticity.

\begin{assumpV}\label{assump:UP2}
	$\sigma:\,[0,1]\to\mathbb{R}^+$ has finite number of points of discontinuity satisfying a~first-order Lipschitz condition except at points of discontinuity.
\end{assumpV}

Our tests will be consistent not only for fixed alternatives with $\delta_n\equiv\delta\neq 0$, but also under local alternative, when the size of the change converges to 0. However, it will only consistently detect changes that are not too small compared to the variance of the partial sum process.

\begin{assumpC}\label{assump:UP3}
	$|\delta_n|\sqrt{n}\to\infty$ as $n\to\infty$.
\end{assumpC}

Henceforth, $\xrightarrow{\prob}$ denotes convergence in probability, $\xrightarrow{\dist}$ convergence in distribution, $\xrightarrow[n\to\infty]{\weak}$ weak convergence in the Skorokhod space $\text{D}[0,1]$ of c\`adl\`ag functions on $[0,1]$, and $[x]$ denotes the integer part of the real number~$x$.

\subsection{Asymptotic distribution of the test statistics}
Under the null hypothesis and the technical assumptions from the previous subsection, the test statistics defined in~\eqref{eq:statistic1} and~\eqref{eq:statistic2} converge to \emph{non-degenerate limit distributions} (their quantiles can be found in Subsection~\ref{subsec:crit}).

\begin{theorem}[Under the null]\label{theorem:undernull}
	Under Assumptions~\ref{assump:UP1}, \ref{assump:UP2}, and under the null hypothesis $\mathcal{H}_0$,
	\begin{equation}\label{eq:limit_dist}
	\mathscr{Q}(V_n)\xrightarrow{\dist}\mathscr{S}(W_{\eta})\quad\mbox{and}\quad\mathscr{R}(V_n)\xrightarrow{\dist}\mathscr{T}(W_{\eta}),\quad n\to\infty,
	\end{equation}
	where~$W_{\eta}(t):=W(\eta(t))$, $\{W(t),0\leq t\leq 1\}$ is a~standard Wiener process, $\eta(t):=\frac{\int_{0}^t\sigma^2(s)\ud s}{\int_{0}^1\sigma^2(s)\ud s}$, and the functionals~$\mathscr{S}$ and~$\mathscr{T}$ are defined in~\eqref{functionalS} and~\eqref{functionalT}.
\end{theorem}

The null hypothesis is rejected at significance level $\alpha$ for large values of $\mathscr{Q}(V_n)$ and $\mathscr{R}(V_n)$. The critical values can be obtained as the $(1-\alpha)$-quantiles of the asymptotic distributions from~\eqref{eq:limit_dist}, if $\eta$ is known. Furthermore, the tests based on these two statistics are consistent, as the test statistics converge to infinity under the alternative, provided that the size of the change does not convergence to~$0$ to fast (see Assumption~\ref{assump:UP3}).

\begin{theorem}[Under the alternative]\label{theorem:underalt}
	Suppose Assumptions~\ref{assump:UP1}, \ref{assump:UP2}, and~\ref{assump:UP3} hold. Under the alternative hypothesis $\mathcal{H}_1$ such that $\tau_n=[n\zeta]$ for some $\zeta\in(0,1)$,
	\[
	\mathscr{Q}(V_n)\xrightarrow{\prob}\infty\quad\mbox{ and }\quad\mathscr{S}(V_n)\xrightarrow{\prob}\infty,\quad n\to\infty.
	\]
\end{theorem}

Theorem~\ref{theorem:underalt} says that in presence of the structural change in mean, the test statistics \emph{explode above all bounds}. Hence, the procedures are \emph{consistent} and the asymptotic distributions from Theorem~\ref{theorem:undernull} can be used to construct the tests. Although, explicit forms of those distributions are unknown. Therefore in order to obtain the critical values, we have to use either simulations from the limit distributions or resampling methods. For the simulation purposes, one would need to know or to estimate the nuisance function $\eta(t)$. The resampling techniques will help us to avoid and overcome such an~issue.

\subsection{Wild bootstrap}

Wild bootstrap \emph{replications} are defined as
\[
Y_{n,k}^{\star}:=\left(Y_{n,k}-\bar{Y}_{n,1:n}\right)X_{k},\quad k=1,\ldots,n,
\]
where $\{X_{n}\}_{n=1}^{\infty}$ is a~sequence of i.i.d.~random variables having \emph{standard normal} $\mathsf{N}(0,1)$ distribution. Moreover, $\{Y_{n,k}\}_{n=1,k=1}^{\infty,n}$ and $\{X_{n}\}_{n=1}^{\infty}$ are also \emph{independent}. The schematic algorithm of the wild bootstrap can be seen as Procedure~\ref{alg:wildboot}. In general, the wild bootstrap replications should be defined in the following way $Y_{n,k}^*:=\bar{Y}_{n,1:n}+\left(Y_{n,k}-\bar{Y}_{n,1:n}\right)X_{k}$, $k=1,\ldots,n$. However, in our case how the test statistics are defined, there would be no impact by adding $\bar{Y}_{n,1:n}$. We define
\[
V_n^{\star}(k):=\sum_{i=1}^{k}Y_{n,i}^{\star}\quad\mbox{and}\quad \widetilde{V}_n^{\star}(k):=V_n^{\star}(n)-V_n^{\star}(k).
\]

\begin{algorithm}[!htb]
	\caption{Wild bootstrap of the test statistic $\mathscr{Q}(V_n)$ and $\mathscr{R}(V_n)$}
	\label{alg:wildboot}
	\begin{algorithmic}[1]
		\REQUIRE Sequence of observations $Y_{1,n},\ldots,Y_{n,n}$ and number of bootstrap replications~$B$
		\ENSURE Bootstrap distributions of $\mathscr{Q}(V_n)$ and $\mathscr{R}(V_n)$, respectively; i.e., the empirical distributions where probability mass $1/B$ concentrates at each of ${}_{(1)}\mathscr{Q}(V_n^{\star}),\ldots,{}_{(B)}\mathscr{Q}(V_n^{\star})$ and ${}_{(1)}\mathscr{R}(V_n^{\star}),\ldots,{}_{(B)}\mathscr{R}(V_n^{\star})$, respectively
		\FOR[repeat in order to obtain the empirical distributions]{$b=1$ to $B$}
		\STATE generate random sample $[{}_{(b)}X_{1},\ldots,{}_{(b)}X_{n}]$ from $\mathsf{N}(0,1)$ independently for different $b$'s 
		\STATE calculate ${}_{(b)}Y_{n,k}^{\star}=\left(Y_{n,k}-\bar{Y}_{n,1:n}\right)\times{}_{(b)}X_{k}$ for all $k$'s
		\STATE calculate ${}_{(b)}V_n^{\star}(k)=\sum_{i=1}^{k}{}_{(b)}Y_{n,i}^{\star}$ and ${}_{(b)}\widetilde{V}_n^{\star}(k)={}_{(b)}V_n^{\star}(n)-{}_{(b)}V_n^{\star}(k)$ for all $k$'s
		\STATE compute the bootstrap test statistics ${}_{(b)}\mathscr{Q}(V_n^{\star})$ and ${}_{(b)}\mathscr{R}(V_n^{\star})$
		\ENDFOR
	\end{algorithmic}
\end{algorithm}

The idea behind bootstrapping is to \emph{mimic the original distribution} of the test statistic in some sense with the distribution of the bootstrap test statistic. It is not known and it does not matter whether our observations come form the null hypothesis or the alternative. We are going to prove that $\mathscr{Q}(V_n^{\star})$ and $\mathscr{R}(V_n^{\star})$, respectively, provide asymptotically correct critical values for the test based on $\mathscr{Q}(V_n)$ and $\mathscr{R}(V_n)$, respectively.

\begin{theorem}[Wild bootstrap validity]\label{theorem:boot}
	Suppose that Assumptions~\ref{assump:UP1} and~\ref{assump:UP2} hold. Under the null hypothesis $\mathcal{H}_0$ or under local alternatives $\mathcal{H}_1$ with $\delta_n\rightarrow 0$ as $n\rightarrow 0$,
	\begin{equation*}
	\mathscr{Q}(V_n^{\star})\xrightarrow{\dist}\mathscr{S}(W_{\eta})\quad\mbox{and}\quad\mathscr{R}(V_n^{\star})\xrightarrow{\dist}\mathscr{T}(W_{\eta}),\quad n\to\infty
	\end{equation*}
	almost surely conditionally on $\{Y_{n,k}\}_{n=1,k=1}^{\infty,n}$. The functionals~$\mathscr{S}$ and~$\mathscr{T}$ are defined in~\eqref{functionalS} and~\eqref{functionalT}, $W_{\eta}(t):=W(\eta(t))$ with~$\eta(t):=\frac{\int_{0}^t\sigma^2(s)\ud s}{\int_{0}^1\sigma^2(s)\ud s}$.
	
	Under the alternative hypothesis $\mathcal{H}_1$ with $\tau_n=[n\zeta]$ for some $\zeta\in(0,1)$ and having $\delta_n\equiv\delta\neq 0$ fixed, let $\{B(t),0\leq t\leq 1\}$ be a standard Wiener processes independent of~$W$. Then,
	\begin{equation*}
	\mathscr{Q}(V_n^{\star})\xrightarrow{\dist}\mathscr{S}\left(W_{\eta}-\frac{\delta}{\varsigma} B_{\zeta}\right)\quad\mbox{and}\quad\mathscr{R}(V_n^{\star})\xrightarrow{\dist}\mathscr{T}\left(W_{\eta}-\frac{\delta}{\varsigma} B_{\zeta}\right),\quad n\to\infty
	\end{equation*}
	almost surely conditionally on $\{Y_{n,k}\}_{n=1,k=1}^{\infty,n}$ with $\varsigma^2=\int_{0}^1\sigma^2(t)\ud t$ and
	\begin{equation*}
	B_{\zeta}(t)=\left\{
	\begin{array}{ll}
	(1-\zeta)B(t), & t\leq \zeta;\\
	B(\zeta)-\zeta B(t), & t> \zeta.
	\end{array}
	\right.
	\end{equation*}
\end{theorem}

Theorem~\ref{theorem:boot} assures that the asymptotic distribution of the bootstrap test statistics and the limit distribution of the original test statistics \emph{coincide} under the null hypothesis. Thus, the bootstrap tests approximately keep the same level as the original tests based on the asymptotics from Theorem~\ref{theorem:undernull} even without knowing or estimating the nuisance function~$\eta(t)$. Moreover, the limit distribution of the bootstrap test is not changed under local alternatives, so we avoid the power loss that would be caused by overestimation of the long-run variance under the alternative.

Even under fixed alternatives, the distribution of the bootstrap statistics converge to an~almost sure finite limit. In contrast, an~uncorrected kernel estimator for the long-run variance would converge to infinity in this case. Depending on the function $\eta$ and the time of the change represented by~$\zeta$, the quantiles $\mathscr{S}(W_{\eta}-\frac{\delta}{\varsigma} B_{\zeta})$ and $\mathscr{T}(W_{\eta}-\frac{\delta}{\varsigma} B_{\zeta})$, respectively, might be larger or smaller than the quantiles of $\mathscr{S}\left(W_{\eta}\right)$ and $\mathscr{T}\left(W_{\eta}\right)$, respectively, resulting in a~loss or gain of power compared to the use of the asymptotic quantiles from Theorem~\ref{theorem:undernull} (which is only feasible when~$\eta$ is known).

Now, the simulated (empirical) distributions of the bootstrap test statistics can be used to calculate the bootstrap critical values, which will be compared to the values of the original test statistics in order to reject the null or not.

\subsection{Changepoint estimator}
If a~change is \emph{detected}, it is of interest to estimate the time of the change. It is sensible to use
\begin{align*}
\hat{\tau}_n&:=\operatorname{argmax}_{1\leq k \leq n}\frac{\left|\sum_{i=1}^{k}\left(Y_{n,i}-\bar{Y}_{n,1:n}\right)\right|+\left|\sum_{i=n-k+1}^{n}\left(Y_{n,i}-\bar{Y}_{n,1:n}\right)\right|}{\max\limits_{1\leq i \leq k}\left|\sum_{j=1}^{i}\left(Y_{n,j}-\bar{Y}_{n,1:k}\right)\right|+\max\limits_{k< i \leq n}\left|\sum_{j=i}^{n}\left(Y_{n,j}-\bar{Y}_{n,(k+1):n}\right)\right|}\\
&\equiv \operatorname{argmax}_{1\leq k \leq n}\frac{\big|V_n(k)-k/n V_n(n)\big|+\big|\widetilde{V}_n(n-k)-k/nV_n(n)\big|}{\max\limits_{1\leq i \leq k}\big|V_n(i)-i/kV_n(k)\big|+\max\limits_{k< i \leq n}\big|\widetilde{V}_n(i)-(n-i)/(n-k)\widetilde{V}_n(k)\big|}
\end{align*}
as a~\emph{changepoint estimator}. Our next theorem shows that under the alternative, the changepoint $\tau_n$ is consistently estimated by the estimator $\hat{\tau}_n$.

\begin{theorem}[Estimator's consistency]\label{theorem:estimation}
	Suppose Assumptions~\ref{assump:UP1}, \ref{assump:UP2}, and~\ref{assump:UP3} hold. Under the alternative hypothesis $\mathcal{H}_1$ such that $\tau_n=[n\zeta]$ for some $\zeta\in(0,1)$, it holds $\hat{\tau}_n/n\xrightarrow{\prob}\zeta$ as $n\to\infty$.
\end{theorem}

\section{Simulations}\label{sec:simu}

\subsection{Asymptotic critical values}\label{subsec:crit}
The explicit forms of the limit distributions stated in~\eqref{eq:limit_dist} are not known. The critical values for the simplest case $\eta(t)=t$ may be determined by simulations from the limit distributions $\mathscr{S}(W)$ and $\mathscr{T}(W)$ from Theorem~\ref{theorem:undernull}. Theorem~\ref{theorem:underalt} ensures that we reject the null hypothesis for large values of the test statistics. We have simulated the asymptotic distributions~\eqref{eq:limit_dist} for the stationary case (i.e., a~situation when $\eta(t)=t$) by \emph{discretizing} the standard Wiener process and using the relationship of a~random walk to the standard Wiener process. We considered $1000$ as the number of discretization points within $[0,1]$ interval and the number of simulation runs equals to $100000$. In Table~\ref{tab:crit_val}, we present several critical values for the test statistics~$\mathscr{Q}(V_n)$ and~$\mathscr{R}(V_n)$ under stationarity.

\begin{table}[!ht]
	\caption{Simulated critical values corresponding to the asymptotic distributions of the test statistics~$\mathscr{Q}(V_n)$ and~$\mathscr{R}(V_n)$ under the null hypothesis, where $\eta(t)=t$}
	\label{tab:crit_val}
	\begin{center}
		\begin{tabular}{rrrrrr}
			\toprule
			$100(1-\alpha)\%$ & $90\%$ & $95\%$ & $97.5\%$ & $99\%$ & $99.5\%$ \\
			\midrule
			$\mathscr{Q}(V_n)$-based & $1.209008$ & $1.393566$ & $1.571462$ & $1.782524$ & $1.966223$\\
			$\mathscr{R}(V_n)$-based & $5.700222$ & $7.165705$ & $8.807070$ & $10.597625$ & $11.755233$\\
			\bottomrule
		\end{tabular}
	\end{center}
\end{table}

\subsection{Simulation study}
We are interested in the performance of the tests based on the self-normalized test statistics $\mathscr{Q}(V_n)$, $\mathscr{R}(V_n)$ (with $\eta(t)=t$ corresponding to the stationary case) and their wild bootstrap counterparts $\mathscr{Q}(V_n^{\star})$, $\mathscr{R}(V_n^{\star})$ that are completely nuisance parameter free. We focused on the comparison of the \emph{accuracy of critical values} obtained by the wild bootstrap method with the accuracy of critical values obtained by the simulation from the limit distributions.

In Figures~\ref{fig:H0} and~\ref{fig:H1}, one may see \emph{size-power plots} for choices of $n\in\{100,400\}$, $\tau_n\in\{n/4,n/2\}$, and $\delta_n\in\{0.5,1.0\}$ considering the test statistics $\mathscr{Q}(V_n)$, $\mathscr{R}(V_n)$, $\mathscr{Q}(V_n^{\star})$, and $\mathscr{R}(V_n^{\star})$ under the null hypothesis and under the alternative. In Figure~\ref{fig:H0}, the empirical rejection frequency under the null hypothesis (actual $\alpha$-errors) is plotted against the theoretical size (theoretical $\alpha$-errors with $\alpha\in\{1\%,5\%,10\%\}$), illustrating the power of the test. The ideal situation under the null hypothesis is depicted by the straight diagonal dotted line. The empirical rejection frequencies ($1-$($\beta$-errors)) under the alternative (with different changepoints and values of the change) are shown in Figure~\ref{fig:H1}. Under the alternative, the desired situation would be a~steep function with values close to~1. For more details on the size-power plots we may refer, e.g., to~\cite{kir2006}. The error terms $\{\sigma(k/n)\eps_k\}_{k=1}^n$ were simulated as two stationary and two non-stationary time series:
\begin{itemize}
	\setlength\itemsep{-.1em}
	\item \textsf{IID} \ldots~independent and identically distributed random variables;
	\item \textsf{AR(1)} \ldots~autoregressive (AR) process of order one having a~coefficient of autoregression equal $0.3$;
	\item \textsf{AR(1)--AR(1)} \ldots~AR process with the coefficient $0.3$, which realizations are multiplied by $\sqrt{2}$ after the first quarter of the time series (deterministic change of volatility);
	\item \textsf{ARCH(1) inc} \ldots~autoregressive conditional heteroscedasticity (ARCH) process with the second coefficient equal $0.9$, whose realizations are randomly and `increasingly' multiplied (random and in average linearly increasing change of volatility).
\end{itemize}

The standard normal distribution and the Student $t$-distribution with 3~degrees of freedom are used for generating the innovations of the models' errors. All of the processes are standardized such that they have unit variance at the beginning. The non-stationary ones have changing variance such that their variance is doubled at the end. In the simulations of the rejection rates, we used $5000$ repetitions. When bootstrapping, for each sample we used $2000$ bootstrap samples to compute the bootstrap critical values.

\begin{figure}[!ht]
\centering \includegraphics[width=.8\textwidth]{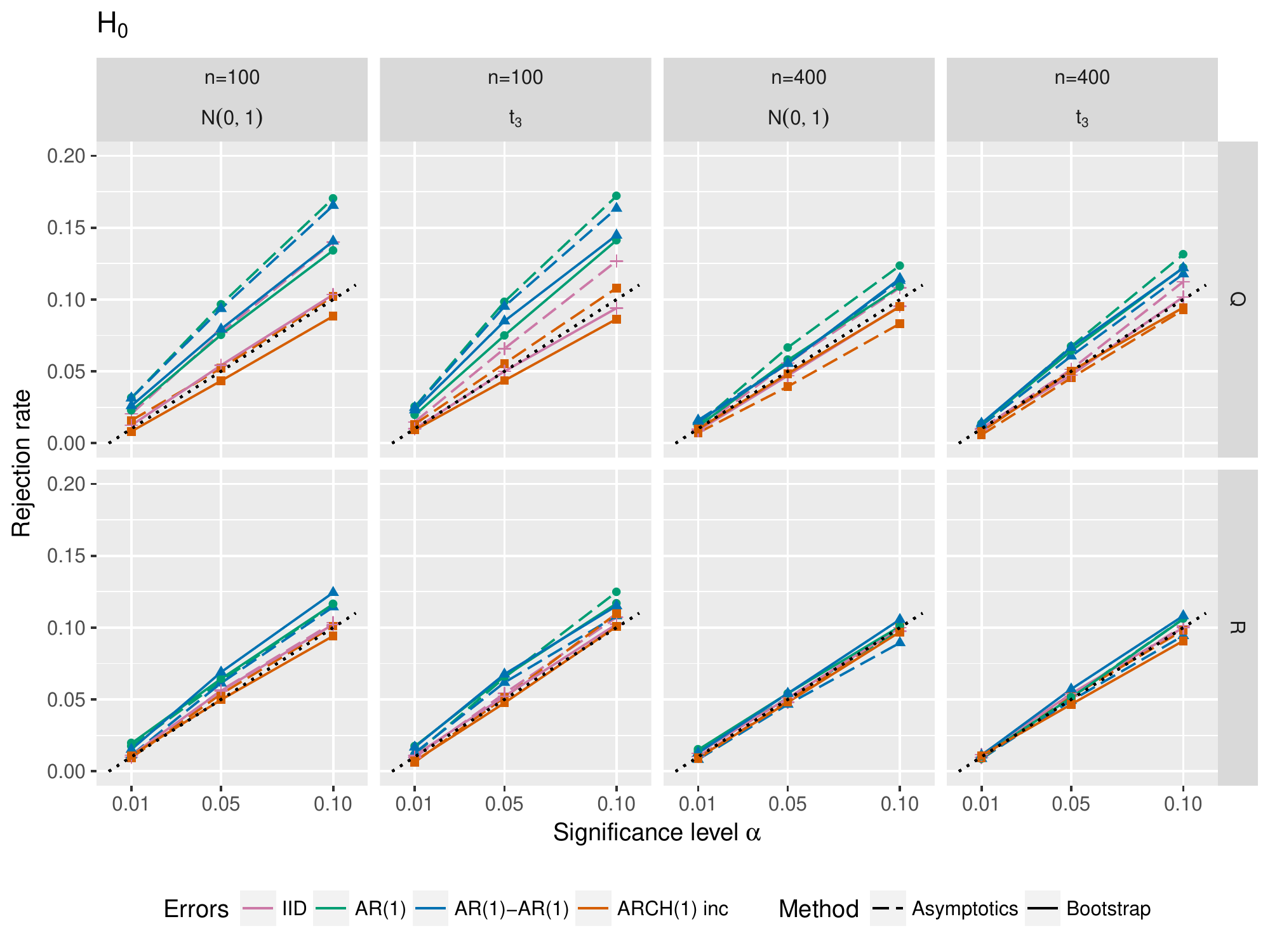}
		\caption{Size-power plots for $\mathscr{Q}$ and $\mathscr{R}$ under $\mathcal{H}_0$}
		\label{fig:H0}
\end{figure}

\begin{figure}[!ht]
\centering \includegraphics[width=.8\textwidth]{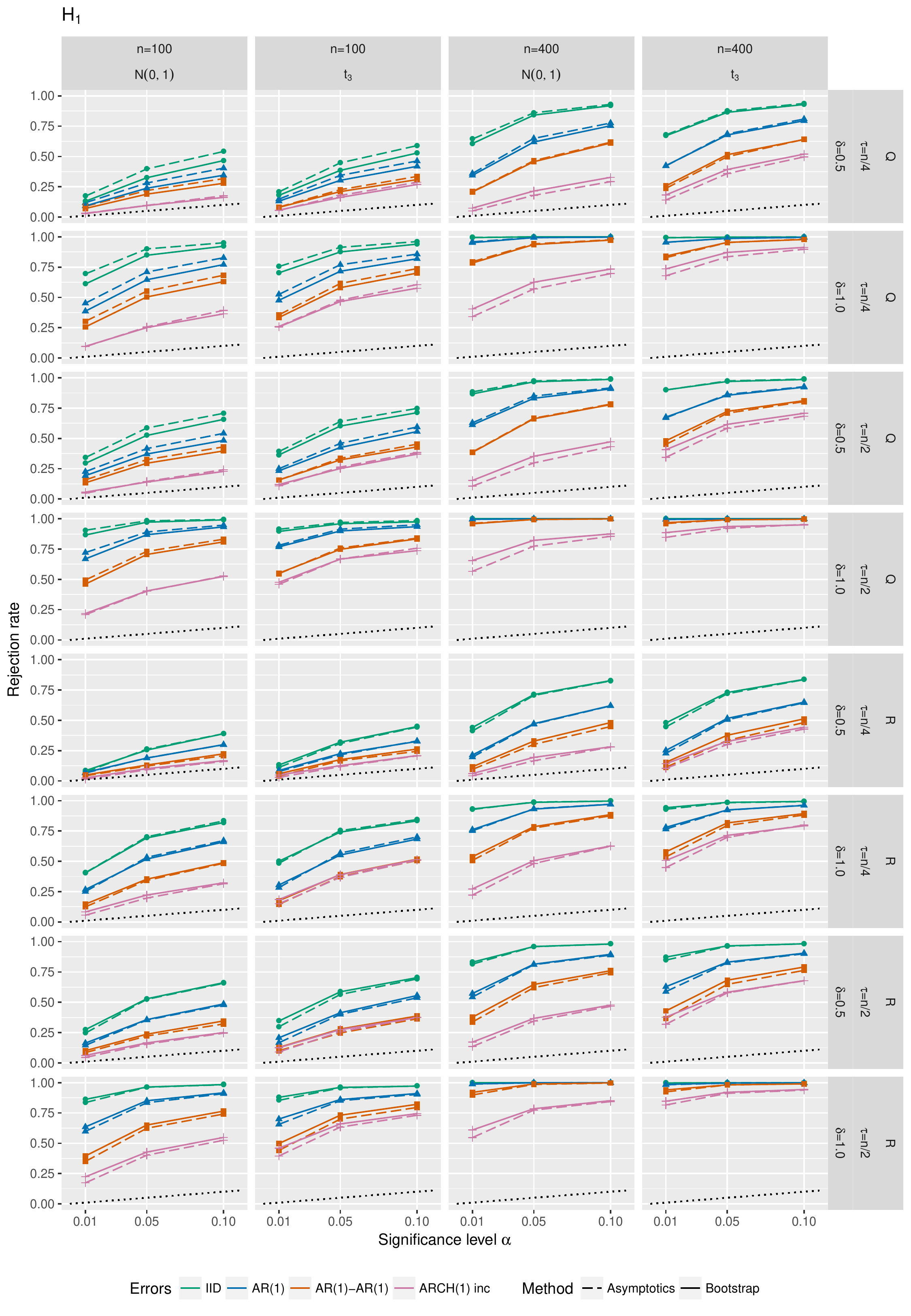}
		\caption{Size-power plots for $\mathscr{Q}$ and $\mathscr{R}$ under $\mathcal{H}_1$}
		\label{fig:H1}
\end{figure}

In all of the subfigures of Figure~\ref{fig:H0} depicting a~situation under the null hypothesis, we may see that comparing the accuracy of $\alpha$-levels (sizes) for different self-normalized test statistics, the integral-type ($\mathscr{R}$-based) method seems to keep the theoretical significance level more firmly than the supremum-type ($\mathscr{Q}$-based) method. The bootstrap approach generally gives critical values that are more accurate than the asymptotic critical values (assuming stationarity, i.e. $\eta(t)=t$), especially for the non-stationary situations. Comparing the case of $\mathsf{N}(0,1)$ innovations with the case of~$t_3$ innovations, the rejection rates under the null tend to be slightly higher for the~$t_3$ distribution. As expected, the accuracy of the critical values tends to be better for larger~$n$.

While the $\mathscr{R}$-method performs better under the null, under the alternative method, it has a tendency to have slightly lower power than the $\mathscr{Q}$-method (see Figure~\ref{fig:H1}). In addition, the wild bootstrap technique provides higher power in some situations besides the fact that they are nuisance parameter free. So we strongly recommend to use this bootstrap method. We may also conclude that under~$\mathcal{H}_1$ with larger abrupt change, the power of the test increases. The power decreases when the changepoint is closer to the beginning or the end of the time series. The heavier tails ($t_3$ against $\mathsf{N}(0,1)$) give worse results in general for both test statistics. Moreover, `more dependent' and `more non-stationary' scenarios reveal worsening of the test statistics' performance.

Additionally, one can use a~size-power plot with the \emph{adjusted} (empirical) $\alpha$-errors to compare the performance of $\mathscr{Q}(V_n)$ against $\mathscr{R}(V_n)$. The \emph{empirical size-power plots} in Figure~\ref{fig:Adjusted} display the empirical size of the test (i.e., $1-$sensitivity) on the $x$-axis versus the empirical power of the test (i.e., specificity) on the $y$-axis. The ideal shape of the curve is as steep as possible. The empirical size-power plots demonstrate that the self-normalized test statistic $\mathscr{Q}(V_n)$ gives approximately the same empirical powers for the adjusted empirical sizes comparing to the test statistic $\mathscr{R}(V_n)$. This is due to two opposing facts: $\mathscr{R}(V_n)$ keeps the significance level of the test better, but $\mathscr{Q}(V_n)$ gives higher power of the test. The very same conclusion can be made for the bootstrap add-ons: When comparing bootstrapping versus asymptotics, the wild bootstrap method gives slightly higher empirical powers for the adjusted empirical sizes compared to the traditional asymptotics assuming underlying stationarity of the time series' disturbances (i.e., $\eta(t)=t$).

\begin{figure}[!ht]
\centering \includegraphics[width=.8\textwidth]{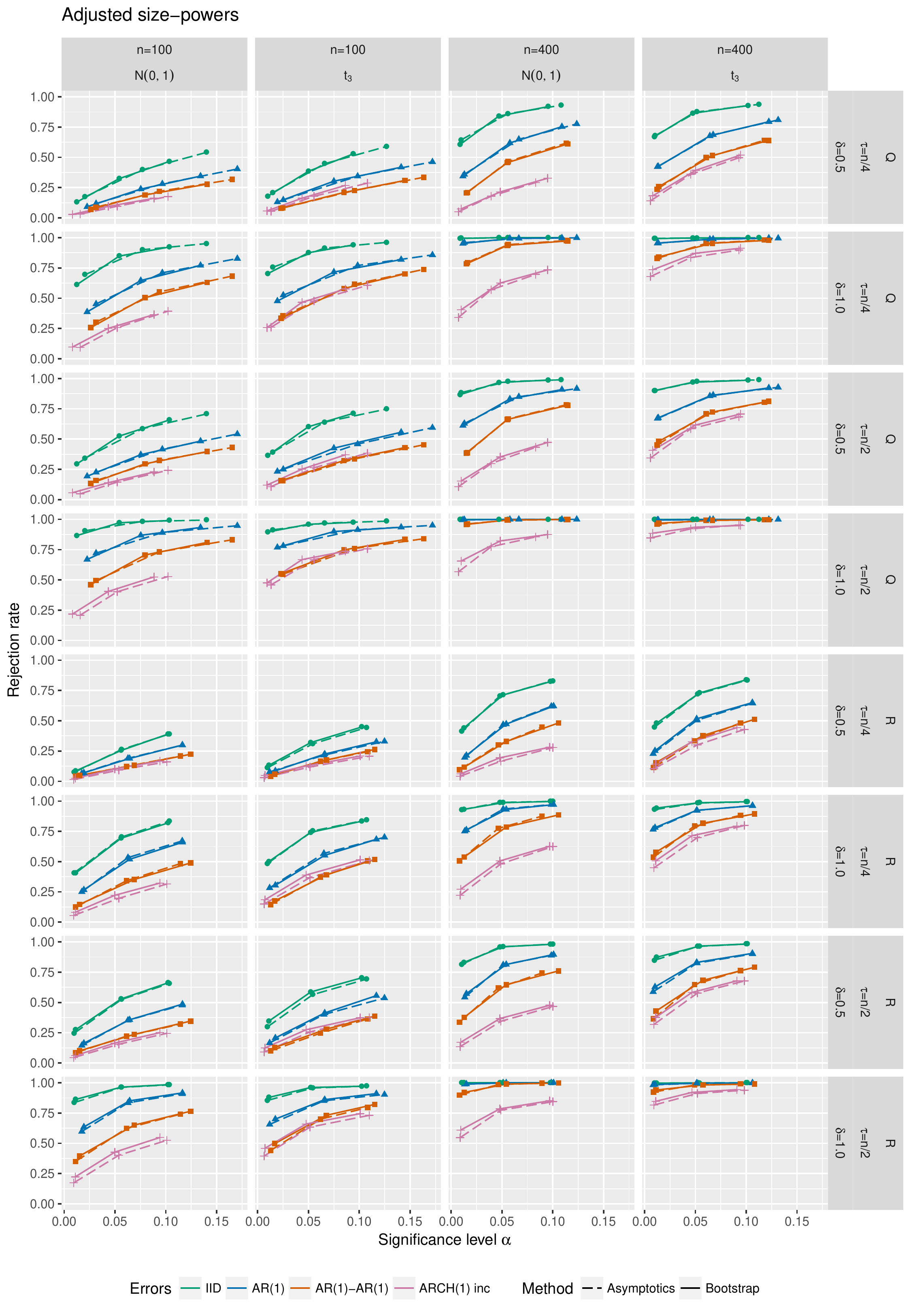}
		\caption{Empirical (adjusted) size-power plots for $\mathscr{Q}$ and $\mathscr{R}$}
		\label{fig:Adjusted}
\end{figure}

Furthermore, a~comparison with a~standard and widely used change point detection procedure is provided with emphasis on computational performance. A~classical representative is the \emph{supremum-type cumulative sums (CUSUM)} test statistic
\begin{equation}\label{eq:nonR_CUSUM}
\mathscr{C}_{M}(V_n):=\frac{1}{\sqrt{\widehat{\sigma}^2_n(M)}}\max_{1\le k\le n-1}\left|\sum_{i=1}^k\left(Y_{n,i}-\bar{Y}_{n,1:n}\right)\right|
\end{equation}
where $\widehat{\sigma}^2_n(M)$ is a~suitable variance estimator. The null hypothesis is rejected for large values of $\mathscr{C}_{M}(V_n)$. For surveys, we refer to, e.g., \cite{Perron2006}.

In order to ensure that a~test statistic is asymptotically distribution-free under the null hypothesis, it is necessary to use a~suitable estimator of variance for the underlying process of random errors. The minimal requirement for $\widehat{\sigma}^2_{n}(M)$ would be consistency under the null hypothesis and boundedness (in probability) under the alternative. Often, the \emph{Bartlett estimator} is used to estimate the variance
\begin{equation*}
\hat\sigma_n^2(M)=\hat{R}(0)+2\sum_{1\le k\le M}\left(1-\frac{k}{M}\right)\hat{R}(k),\quad M<n,
\end{equation*}
where $\hat{R}(k)=\frac{1}{n}\sum_{1\le i\le n-k}(Y_{n,i}-\bar{Y}_{n,1:n})(Y_{n,i+k}-\bar{Y}_{n,1:n})$, $0\le k<n$. However, it does not always provide satisfactory results and finding a~proper value of~$M$ may be troublesome. The rate of convergence is small even under the null hypothesis and $\hat\sigma_n^2(M)$ might go to infinity under the alternative~\citep{HHH2008}. Other similar types of estimators can be used instead, for instance Parzen kernels \citep{Andrews1991}, but they still possess the described deficiency.

The consistency properties of the above described Bartlett estimator and of its modification are studied in~\cite{Antoch1997}. The authors also describe difficulties of long-run variance estimation when detecting a~change in the mean of a~linear process in more detail. A~simulation study shows that it is not easy to find a~variance estimate that would work well both under null hypothesis and under alternative. Furthermore, such estimators are often very sensitive to the \emph{choice of the window length~$M$}. Based on simulation studies performed by~\cite{Antoch1997}, we decided to use a~rule of thumb of $M=n/10$. Asymptotic distribution of the test statistic~\eqref{eq:nonR_CUSUM} can be found, e.g., in~\cite{CH1997}. The corresponding critical values come from~\cite{Kulperger1990}.

Now, we demonstrate the performance of our approaches---the $\mathscr{Q}$ and $\mathscr{R}$ self-normalized test statistics---compared to the traditional CUSUM test statistic (i.e., based on $\mathscr{C}_M$). Empirical sizes under the null hypothesis and empirical powers under the alternative ($\tau_n=n/2$ and $\delta_n=0.5$) of our two detection procedures compared to the standard one are shown in Figure~\ref{fig:H0Compare} and in Figure~\ref{fig:H1Compare}, respectively. The number of repetitions for the simulation of the rejection rates is again set to $5000$. The sample size is chosen as $n=200$ and the window length for the variance estimate is $M=20$.

\begin{figure}[!ht]
\centering \includegraphics[width=.8\textwidth]{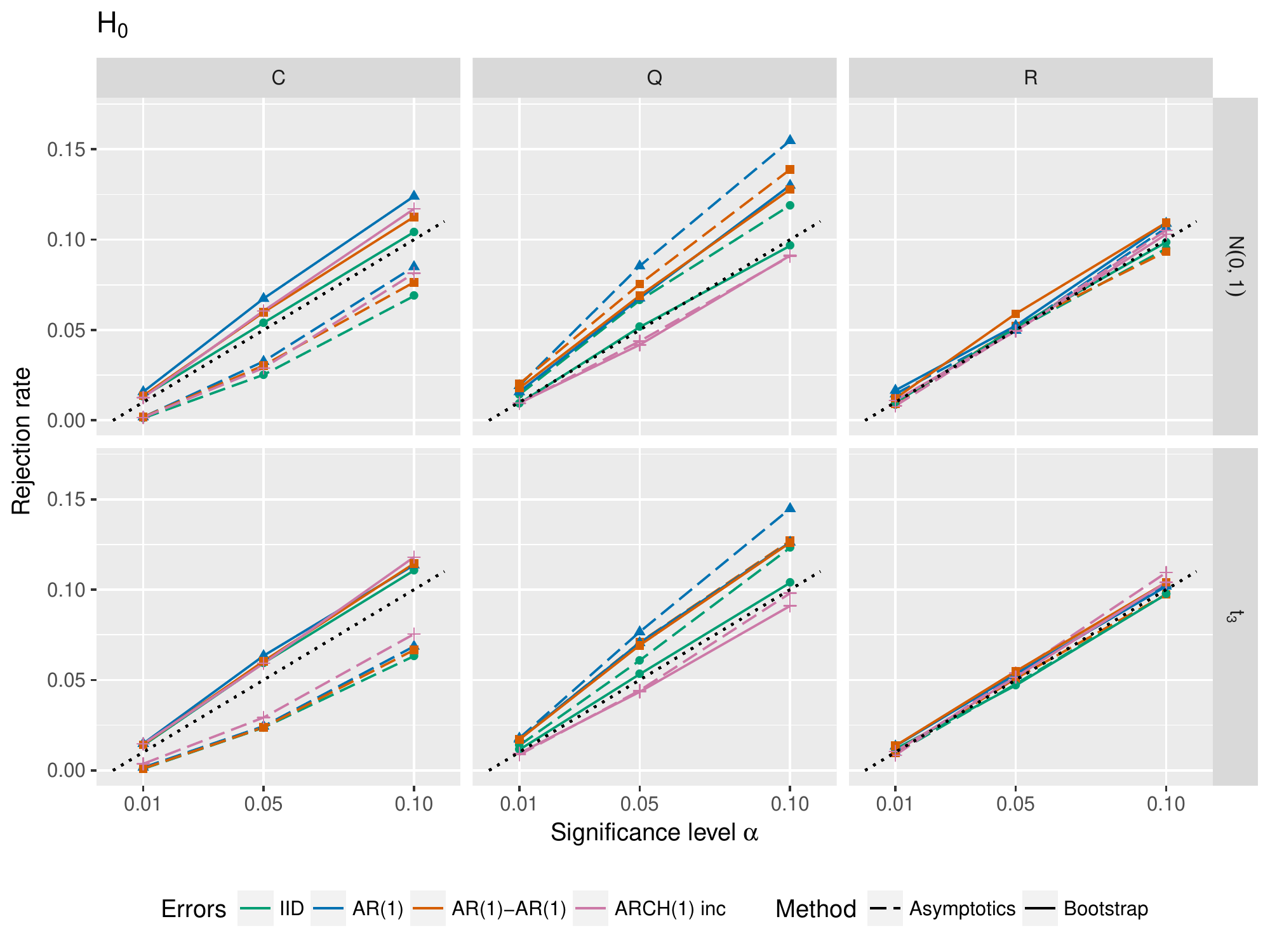}
		\caption{Size-power plots for $\mathscr{C}_M$, $\mathscr{Q}$, and $\mathscr{R}$ under $\mathcal{H}_0$ (sample size $n=200$)}
		\label{fig:H0Compare}
\end{figure}

\begin{figure}[!ht]
\centering \includegraphics[width=.8\textwidth]{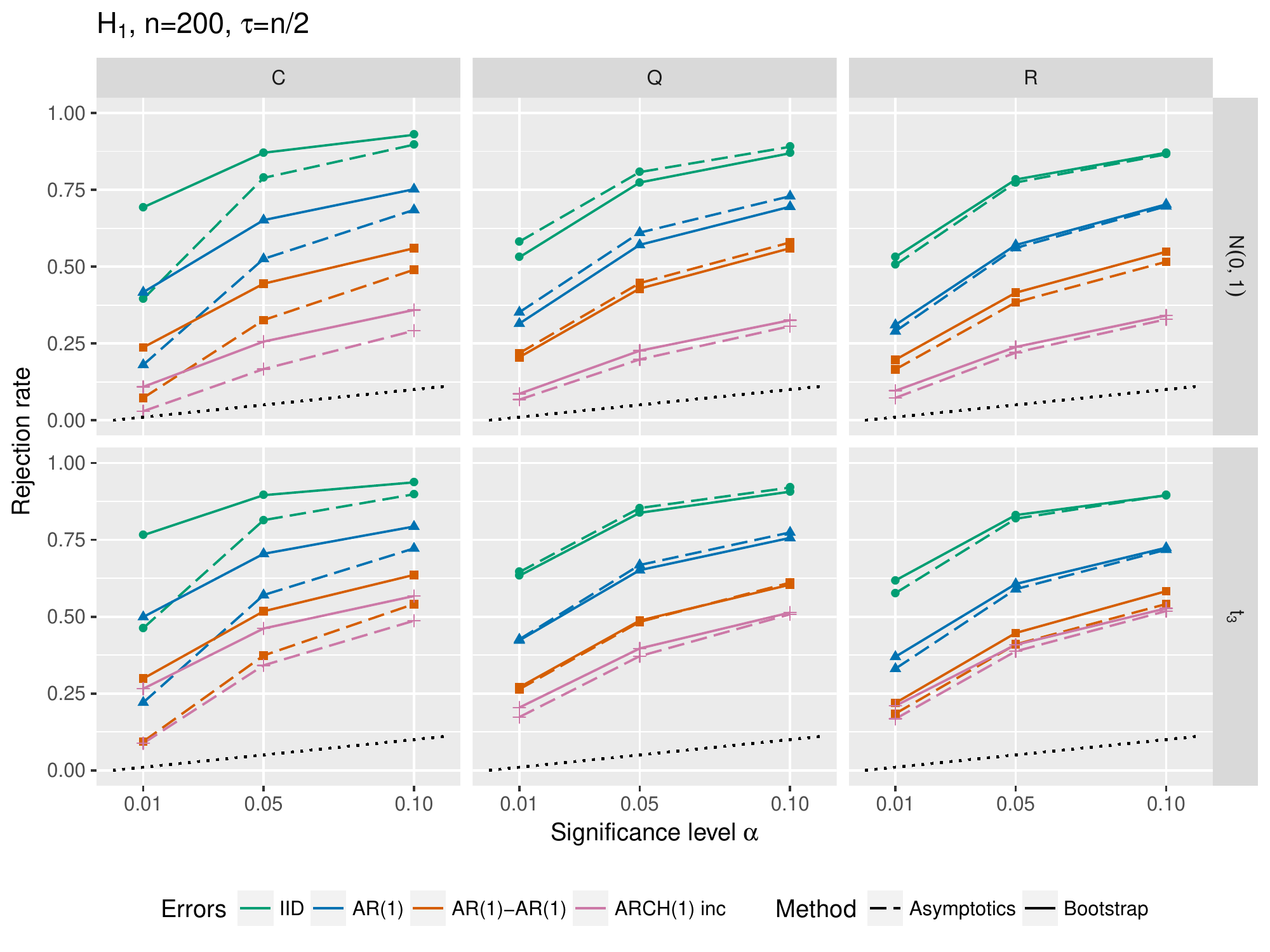}
		\caption{Size-power plots for $\mathscr{C}_M$, $\mathscr{Q}$, and $\mathscr{R}$ under $\mathcal{H}_1$ (sample size $n=200$, change of $\delta_n=0.5$ at time $\tau_n=n/2$)}
		\label{fig:H1Compare}
\end{figure}

\begin{figure}[!ht]
\centering \includegraphics[width=.8\textwidth]{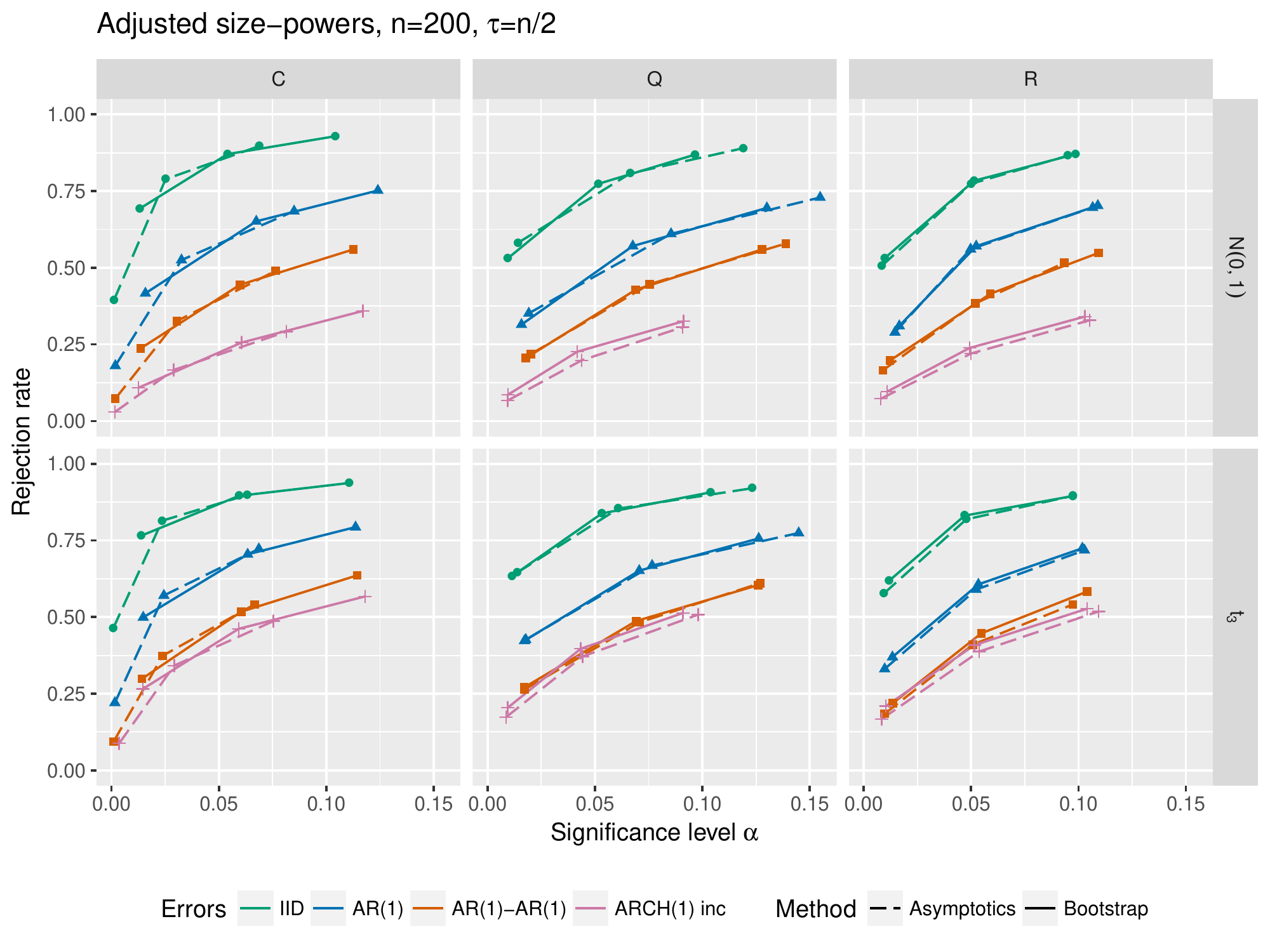}
		\caption{Empirical (adjusted) size-power plots for $\mathscr{C}_M$, $\mathscr{Q}$, and $\mathscr{R}$ (sample size $n=200$, change of $\delta_n=0.5$ at time $\tau_n=n/2$)}
		\label{fig:AdjCompare}
\end{figure}

To conclude, the CUSUM test statistic yields too small empirical size, see Figure~\ref{fig:H0Compare}. It rejects more often than it should and, moreover, it provides lower power (Figure~\ref{fig:H1Compare}) compared to the self-normalized type test procedures, especially for small significance levels ($5\%$ and $1\%$). There are two possible reasons for that: the classical CUSUM procedure relies on the variance estimate, which can be troublesome, and it requires a~suitable choice of the nuisance parameter. This illustrates that avoiding the nuisance parameter estimation should really be considered as one of advantages of the proposed methods. Besides that, the wild bootstrap performs better compared to the traditional asymptotics, which can be illustrated via adjusted size-power plots in Figure~\ref{fig:AdjCompare}.

Furthermore, one can concentrate on a~situation that is far away from a~stationary case. Let us take into consideration a~zero-mean AR(1) sequence (the AR-coefficient is set to $0.3$) of $n=200$ random errors, where the random variables from the first quarter of the series are multiplied by~$10$. This leads to the variance function $\eta(t)=40t/13$ for $t\in[0,1/4)$ and $\eta(t)=10/13+(4t-1)/13$ for $t\in[1/4,1]$. The corresponding alternative $\mathcal{H}_1$ is chosen as $\tau_n=n/2$ and $\d_n=1$. Figure~\ref{fig:Extreme} evidence, on one hand, that the asymptotic approach assuming underlying stationarity yields very unreliable results. On the other hand, the bootstrap method provides reasonable rejection rates even in such non-stationary case.




\begin{figure}[!ht]
	\begin{center}
		\includegraphics[width=0.49\textwidth]{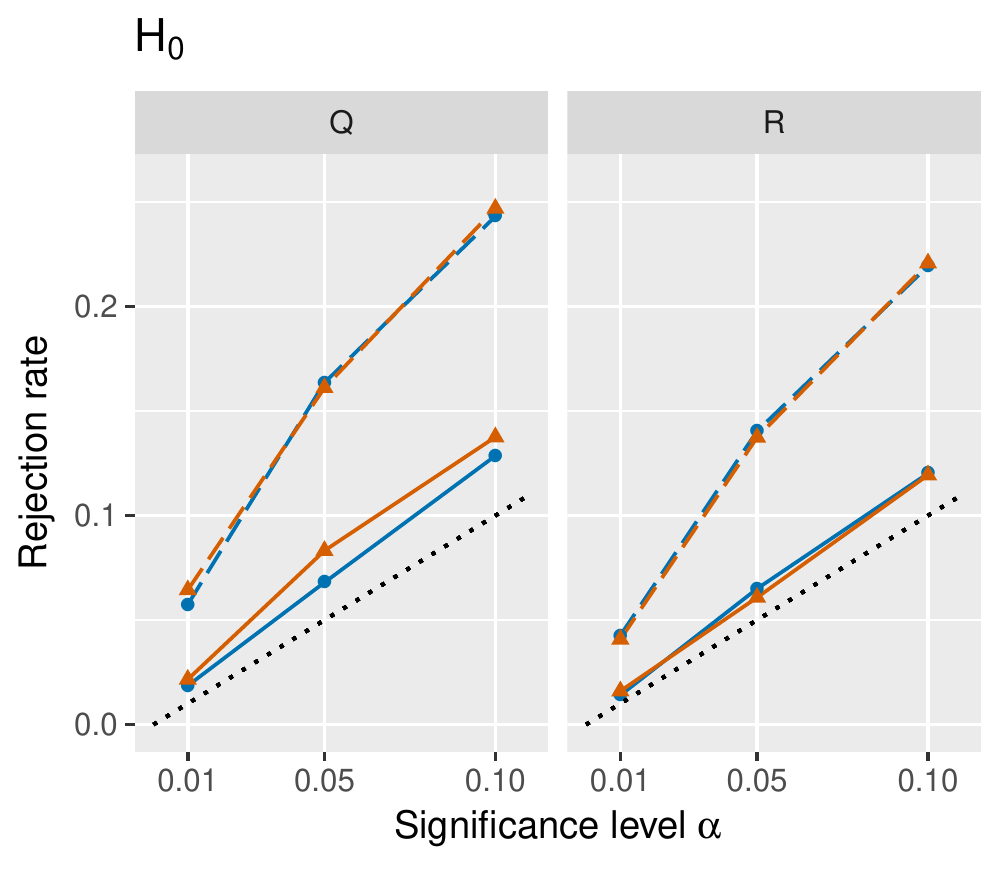}
		\includegraphics[width=0.49\textwidth]{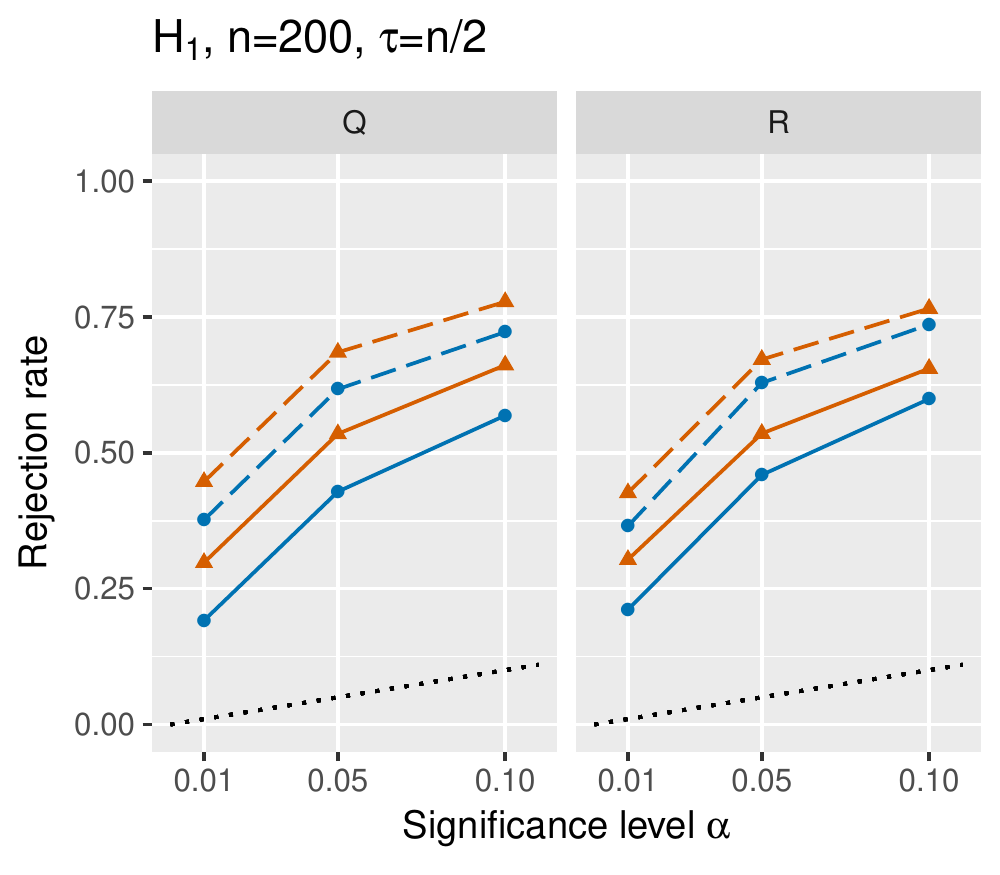}
		\includegraphics[width=0.61\textwidth]{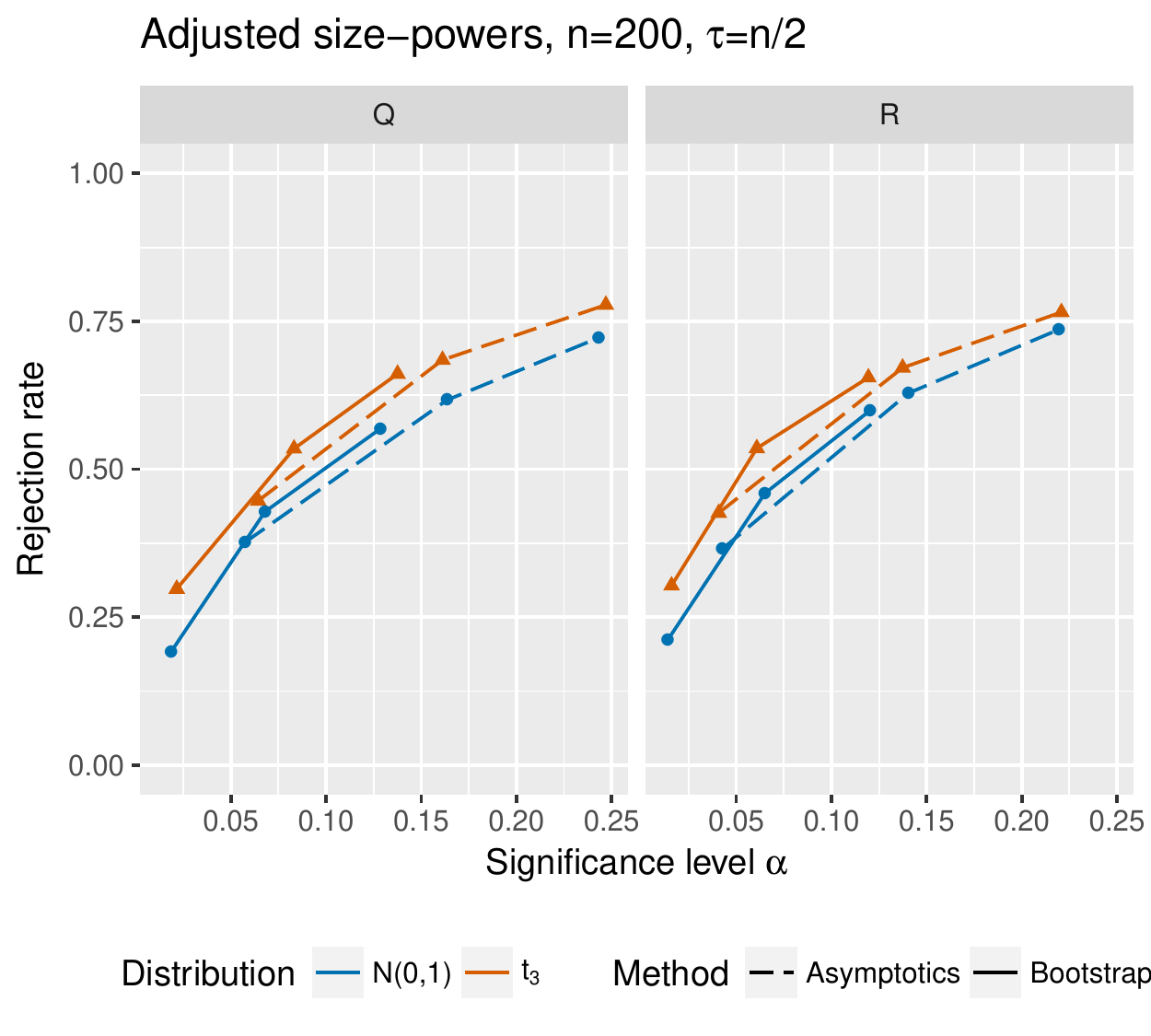}
	\end{center}
		\caption{Size-power plots for $\mathscr{Q}$ and $\mathscr{R}$ under $\mathcal{H}_0$ (top-left) and under $\mathcal{H}_1$ (top-right) for a~very heteroscedastic case together with the corresponding empirical (adjusted) size-power plots (bottom)}
		\label{fig:Extreme}
\end{figure}

Afterwards, a~simulation experiment is performed to study the \emph{finite sample} properties of the changepoint estimator for an~abrupt change in the mean. In particular, the interest lies in the \emph{empirical distributions} of the proposed estimator visualized via boxplots, see Figure~\ref{fig:Estimator}. The simulation setup is kept the same as described above.

\begin{figure}[!ht]
\centering \includegraphics[width=.8\textwidth]{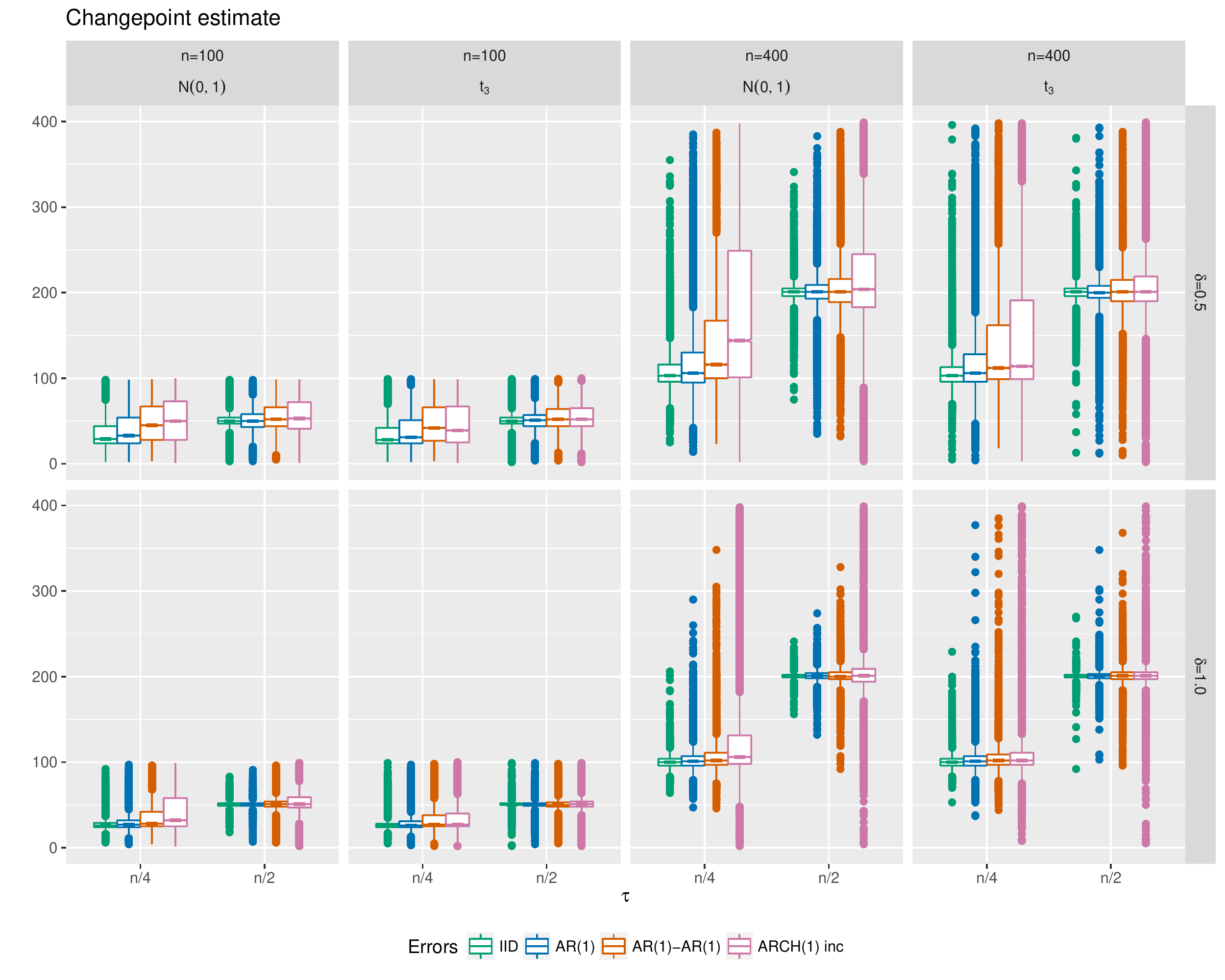}
		\caption{Boxplots of the estimated changepoint $\hat{\tau}_n$}
		\label{fig:Estimator}
\end{figure}

It can be concluded that the precision of our changepoint estimator is satisfactory even for relatively short time series regardless of the errors' structure. Furthermore, the disturbances with heavier tails or changing variance yield less precise estimators than stationary innovations with light tail. One may notice that higher precision is obtained when the changepoint is closer to the middle of the time series. It is also clear that the precision of $\hat{\tau}_n$ improves markedly as $\delta_n$ increases.

\section{Practical Applications}\label{sec:data}

\subsection{Dieselgate}
Especially in many time series from finance, (conditional) heteroscedasticity appears frequently. In our first data example, we analyze the daily absolute log returns of the Volkswagen stock prices from January~1, 2015 to November~26, 2015 (VOW.DE, XETRA -- XETRA Delayed Price. Currency in EUR. Open. Downloaded on May 30, 2018 \url{https://finance.yahoo.com/quote/VOW.DE?p=VOW.DE}), which are visualized in Figure~\ref{fig:VW}.

\begin{figure}[!ht]
\centering \includegraphics[width=.8\textwidth]{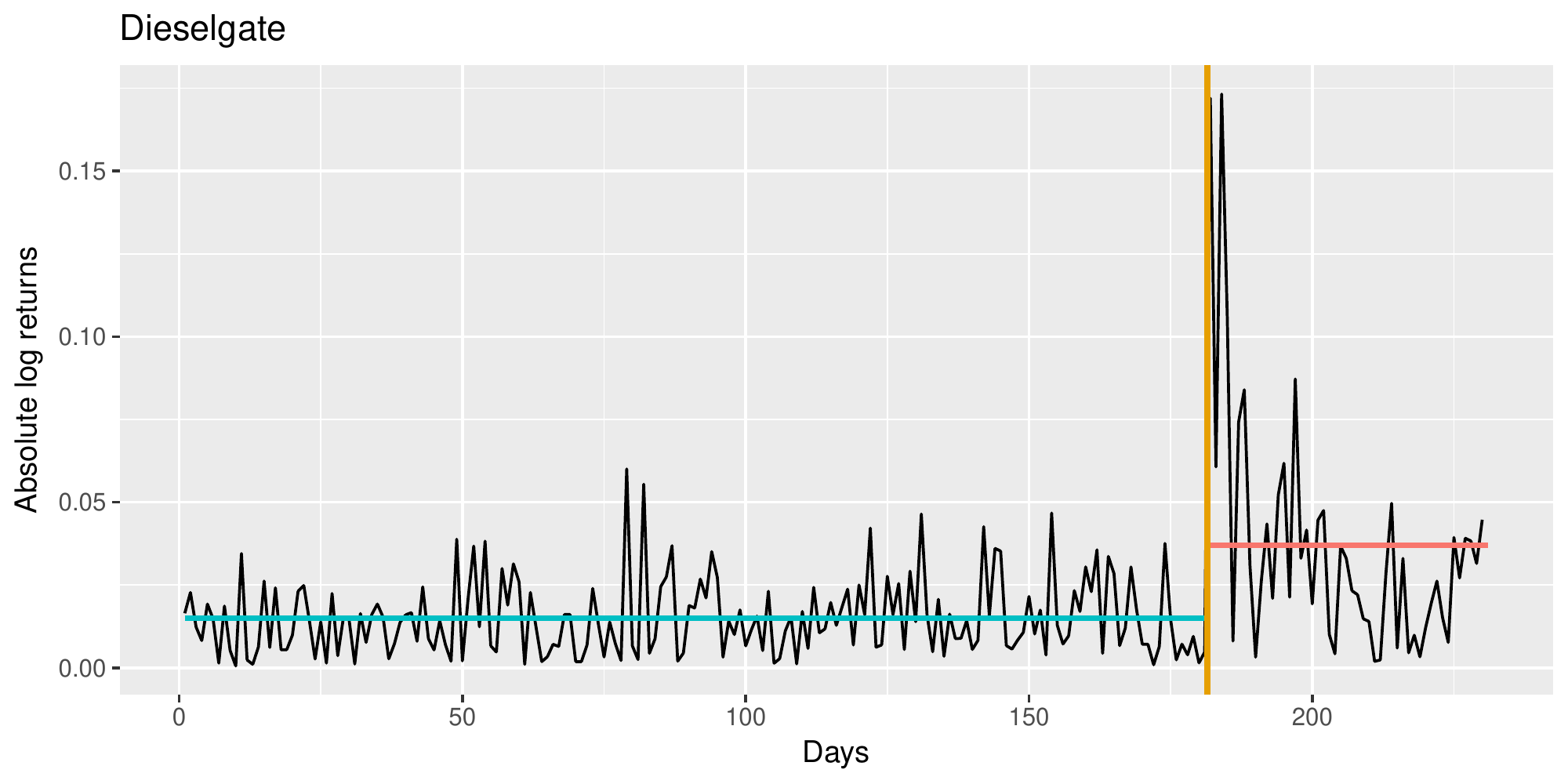}
		\caption{Absolute log returns of the Volkswagen stock price (January~1, 2015 -- November~26, 2015). The changepoint estimate corresponding to the emissions scandal on September~18, 2015 is depicted by the vertical line}
		\label{fig:VW}
\end{figure}

Both our tests as well as their wild bootstrap add-ons reject the null hypothesis of a~constant mean in the absolute log returns (cf.~Table~\ref{tab:VW}), indicating an~increased volatility of the Volkswagen stock price. In contrast, \cite{dehling2015robust} did not find a~significant change using the classical CUSUM-test. In their article, several robust tests detected a~change. There are several large values in this time series, but the reason could be a~period with strongly increased variance.

\begin{table}[!ht]
	\caption{Self-normalized test statistics (asymptotic and bootstrap) together with the corresponding critical values for the Volkswagen stock price data, considering a~significance level of $5\%$}
	\label{tab:VW}
	\begin{center}
		\begin{tabular}{rrrrr}
			\toprule
			& $\mathscr{Q}(V_n)$ & $\mathscr{R}(V_n)$ & $\mathscr{Q}(V_n^{\star})$ & $\mathscr{R}(V_n^{\star})$ \\
			\midrule
			Test statistic & $1.546779$ & $10.74026$ & $1.546779$ & $10.74026$\\
			Critical value & $1.393566$ & $7.165705$ & $1.388683$ & $8.109334$\\
			\bottomrule
		\end{tabular}
	\end{center}
\end{table}

As an estimator for our change, we obtain $\hat{\tau}_n=182$ (depicted by a~vertical line in Figure~\ref{fig:VW}), which corresponds to September 18, 2015. On this day, the United States Environmental Protection Agency issued a notice of violation, which lead to the Volkswagen emissions scandal. Our procedure is capable to detect and, consequently, to estimate the changepoint based on only 10~weeks of daily data after the emissions scandal.

\subsection{Elbe river}
Our second data example consists of the annual maximum discharge of the river Elbe at Dresden, Germany, in the years~1851 to~2012. The variance seems to be lower in the 20th century compared to second half of the 19th century (see Figure~\ref{fig:Labe}). Therefore, we think that our tests are a~good choice for this data set, as they are not effected by heteroscedasticity.

\begin{figure}[!ht]
\centering \includegraphics[width=.8\textwidth]{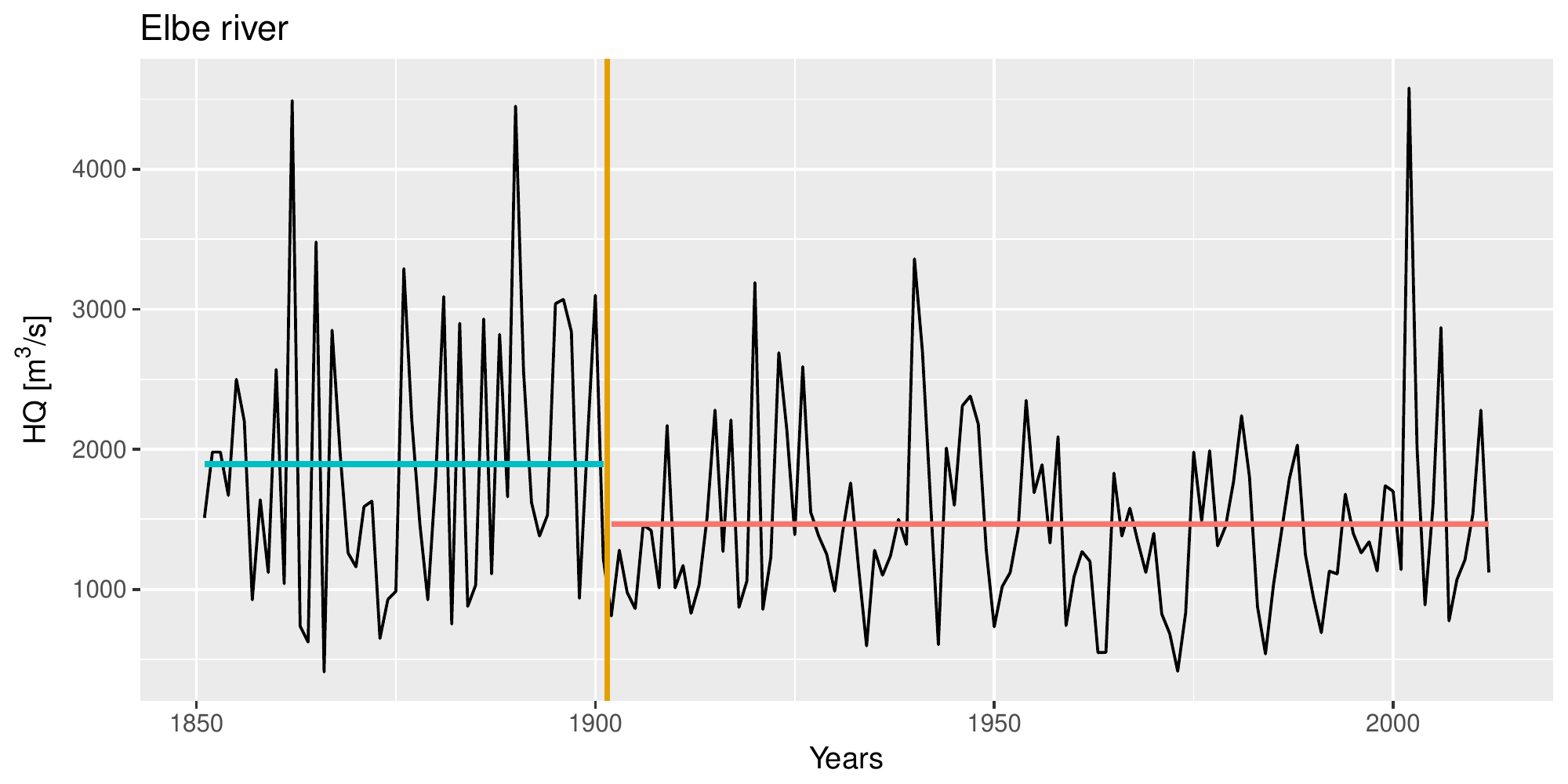}
\caption{Yearly maximal discharges (1851 -- 2012) of the Elbe river in Dresden, Germany. Year~1901 depicted by the vertical line is the changepoint estimate}
\label{fig:Labe}
\end{figure}

All of our four testing procedures reject the null hypothesis of constant mean of the maximum discharges over the whole observation period, see Table~\ref{tab:Labe}. This data set has been previously analyzed with other methods: \cite{sharipov2016sequential} used a~Cram\'er-von Mises-type test statistic and detected a~change in the marginal distribution. \cite{vogel2017studentized} detected a~shift in location using a~robust test based on the Hodges-Lehmann-estimator, while the usual CUSUM statistic did not lead to rejection of the hypothesis (stationarity). So, our self-normalized tests seem to work more reliable than the ordinary CUSUM-test also in this example.

\begin{table}[!ht]
	\caption{Self-normalized test statistics (asymptotic and bootstrap) together with the corresponding critical values for the annual maximum discharge of the river Elbe at Dresden, Germany, considering a~significance level of $5\%$}
	\label{tab:Labe}
	\begin{center}
		\begin{tabular}{rrrrr}
			\toprule
			& $\mathscr{Q}(V_n)$ & $\mathscr{R}(V_n)$ & $\mathscr{Q}(V_n^{\star})$ & $\mathscr{R}(V_n^{\star})$ \\
			\midrule
			Test statistic & $1.481084$ & $7.936363$ & $1.481084$ & $7.936363$\\
			Critical value & $1.393566$ & $7.165705$ & $1.476675$ & $7.599609$\\
			\bottomrule
		\end{tabular}
	\end{center}
\end{table}

\section{Conclusions}\label{sec:con}

We have proposed two tests for changepoints with desirable theoretical properties: The asymptotic size of the tests is guaranteed by a~limit theorem even under heteroscedasticity and dependence, the tests and the related changepoint estimator are consistent. By combining self-normalization and the wild bootstrap, there are neither tuning nor nuisance parameters involved in the whole testing procedure, which makes this framework effortlessly applicable. In our simulations, the tests show reliable performance. Especially the bootstrap test based on the integral-type self-normalized CUSUM-statistic~$\mathscr{R}$ has an~empirical size very close to the nominal level in a~wide range of situations. In the data examples, we have shown that our tests can find changes, which were not detected before using the ordinary CUSUM-tests.

Let us note that the test statistic could also be applied for other data generating processes. The limit distribution is derived from the limit distribution of the partial sum process by the continuous mapping theorem. \cite{Shao2011} studied long range dependent process, where the partial sum process converges weakly to a~fractional Brownian motion, and it would be possible to obtain the limit distribution of our new test statistic in the same way. Furthermore in the case of heavy-tailed random variables, the partial sum process might converge weakly to a~stable L\'evy process. While it should be possible to identify the limit distribution of our test statistics with the help of the continuous mapping theorem, we would expect a~loss of power. Another possibility would be robustified tests following ideas of \cite{PP2018}. But this goes beyond the scope of this paper and is a~topic for future research.








\appendix

\section{Appendix}

\subsection{Auxiliary lemmas and corollaries}
\begin{lemma}\label{lem:continuity}
Let $w\in\mathsf{D}[0,1]$ be a~continuous function satisfying $\nexists a,b:\,w(t)=a+bt,\,\forall t\in[0,1/2]$ and $\nexists a,b:\,w(t)=a+bt,\,\forall t\in[1/2,1]$. Then the following two mappings $\mathsf{D}[0,1]\rightarrow \mathsf{D}[0,1]$ are continuous in $w$:
	\begin{align}
	w&\mapsto\left(\frac{w(t)-tw(1)}{\sup_{s\in[0,t]}\big|w(s)-\frac{s}{t}w(t)\big|+\sup_{s\in[t,1]}\big|w(1)-w(s)-\frac{1\!-\!s}{1\!-\!t}\left(w(1)-w(t)\right)\big|}\right)_{t\in [0,1]};\label{stat1}\\
	w&\mapsto\left(\frac{w(t)-tw(1)}{\left\{\int_0^t\big(w(s)-\frac{s}{t}w(t)\big)^2 \ud s+\int_t^1\big(w(1)-w(s)-\frac{1\!-\!s}{1\!-\!t}\left(w(1)-w(t)\right)\big)^2\ud s\right\}^{1/2}}\right)_{t\in [0,1]}.\label{stat2}
	\end{align}
\end{lemma}

\begin{proof}
	We will first show that the denominator is uniformly bounded away from~$0$. We will only treat the case $t\geq 1/2$, the other case can be treated in the same way. For the mapping in~\eqref{stat1}, we have that
	\begin{equation*}
	\sup_{s\in[0,t]}\left|w(s)-\frac{s}{t}w(t)\right|\geq \sup_{s\in[0,1/2]}\left|w(s)-\frac{s}{t}w(t)\right|\geq \inf_{a,b\in\R}\sup_{s\in[0,1/2]}|w(s)-a-bs|>0,
	\end{equation*}
	because of our assumptions on $w$. For the mapping in~\eqref{stat2}, with similar arguments we get
	\[
	\int_0^t\left(w(s)-\frac{s}{t}w(t)\right)^2\ud s\geq \int_0^{1/2}\left(w(s)-\frac{s}{t}w(t)\right)^2\ud s\geq \inf_{a,b\in\R}\int_0^{1/2}(w(s)-a-bs)^2\ud s>0.
	\]
	Next recall that for continuous $w$, convergence to $w$ in $D[0,1]$ is equivalent to uniform convergence \citep[p.~112]{Billingsley1968}. If $\|w-\tilde{w}\|<\epsilon$, then for all $t\in[0,1]$ the numerator and the denominator of~\eqref{stat1} and~\eqref{stat2} will differ by at most $\epsilon$, which establishes the continuity of the two mappings.
\end{proof}


The following functionals of the partial sum process
\[
U_n(t):=\frac{1}{\sqrt{n}}\sum_{i=1}^{[nt]}\sigma(i/n)\eps_i
\]
can be regarded as continuous modifications of our test statistic: Let
\begin{equation}\label{functionalS}
\mathscr{S}(U_n):=\sup_{0\leq t \leq 1}\left|\frac{U_n(t)-tU_n(1)}{\sup\limits_{0\leq s \leq t}\big|U_n(s)-s/tU_n(t)\big|+\sup\limits_{t\leq u \leq 1}\big|\widetilde{U}_n(u)-(1-u)/(1-t)\widetilde{U}_n(t)\big|}\right|,
\end{equation}
where $\widetilde{U}_n(t):=U_n(1)-U_n(t)$. Moreover,
\begin{equation}\label{functionalT}
\mathscr{T}(U_n):=\int_0^1\frac{\big\{U_n(t)-tU_n(1)\big\}^2}{\int_0^t\big\{U_n(s)-s/tU_n(t)\big\}^2\ud s+\int_t^1\big\{\widetilde{U}_n(u)-(1-u)/(1-t)\widetilde{U}_n(t)\big\}^2\ud u}\ud t.
\end{equation}
The continuity of the functionals~$\mathscr{S}$ and~$\mathscr{T}$ follows directly from Lemma~\ref{lem:continuity}.

\begin{corollary}\label{corollary:limitdist}
	Under Assumptions~\ref{assump:UP1} and~\ref{assump:UP2}, $\mathscr{S}(U_n)\xrightarrow{\dist}\mathscr{S}(W_{\eta})$ and $\mathscr{T}(U_n)\xrightarrow{\dist}\mathscr{T}(W_{\eta})$ as $n\to\infty$.
\end{corollary}

\begin{proof}
	Assumptions~\ref{assump:UP1} and~\ref{assump:UP2} together with the Lemmas~1 and~2 by~\cite{Cav2005} provide the assertion of a~functional central limit theorem, i.e.,
	\[
	\frac{1}{\varsigma\lambda\sqrt{n}}\sum_{i=1}^{[nt]}\sigma(i/n)\eps_i\xrightarrow[n\to\infty]{\weak} W(\eta(t)),
	\]
	where $\varsigma^2=\int_{0}^1\sigma^2(t)\ud t$ and $\eta(t)=\int_{0}^t\sigma^2(s)\ud s/\varsigma^2$. The term $\varsigma\lambda$ is present in the numerator as well as in the denominator of the self-normalized test statistics, which can be canceled out. Afterwards, Lemma~\ref{lem:continuity} completes the proof.
\end{proof}

We are now going to show that the self-normalized test statistic $\mathscr{Q}(V_n)$ under the null behaves asymptotically like~$\mathscr{S}(U_n)$. Analogously, it is going to be demonstrated that the self-normalized test statistic $\mathscr{R}(V_n)$ under the null has the same asymptotic distribution as~$\mathscr{T}(U_n)$.

\begin{proposition}\label{prop:discreteapprox1}
	Assume that there is a~sequence $\{b_n\}_{n\in\N}$ such that
	\begin{multline}\label{assuprop1}
	\prob\bigg[\min_{1\leq k\leq n}\bigg\{\max_{1\leq i \leq k}\Big|\sum_{j=1}^{i}\Big(\sigma_n(j/n)\eps_{j}-\frac{1}{k}\sum_{\ell=1}^k\sigma(\ell/n)\eps_{\ell}\Big)\Big|\bigg.\bigg.\\
	\bigg.\bigg.+\max_{k< i \leq n}\Big|\sum_{j=i}^{n}\Big(\sigma_n(j/n)\eps_{j}-\frac{1}{n-k}\sum_{\ell=k+1}^n\sigma(\ell/n)\eps_{\ell}\Big)\Big|\bigg\}\geq b_n\bigg]\rightarrow 1
	\end{multline}
	and 
	\begin{equation}\label{assuprop2}
	\max_{1\leq k\leq n}|\eps_{k}|=\op (b_n).
	\end{equation}
	Then, under the null hypothesis $\mathcal{H}_0$, $\mathscr{Q}(V_n)-\mathscr{S}(U_n)\xrightarrow[n\to\infty]{\prob}0$.
\end{proposition}

\begin{proof}
	First note that $\frac{1}{\sqrt{n}}\sum_{i=1}^k\Big(\sigma_n(i/n)\eps_{i}-\frac{1}{n}\sum_{j=1}^n\sigma(j/n)\eps_{j}\Big)=U_n(k/n)-\frac{k}{n}U_n(1)$. Due to Assumption~\ref{assump:UP2}, we get $\sup_{t\in[0,1]}\sigma(t)=:M<\infty$. For $t\in[0,1]$, let $k=[tn]$, then
	\[
	\bigg|\sum_{i=1}^k\Big(\sigma_n(i/n)\eps_{i}-\frac{1}{n}\sum_{j=1}^n\sigma(j/n)\eps_{j}\Big)-\sqrt{n}\left(U_n(t)-tU_n(1)\right)\bigg|\leq \frac{1}{\sqrt{n}}|U_n(1)|\leq M\max_{1\leq i\leq n}|\eps_{i}|.
	\]
	For $s\in[0,t]$, let $j=[sn]$. Hence,
	\begin{multline*}
	\bigg|\sum_{i=1}^j\Big(\sigma_n(i/n)\eps_{i}-\frac{1}{k}\sum_{\ell=1}^k\sigma(\ell/n)\eps_{\ell}\Big)-\sqrt{n}\left(U_n(s)-\frac{s}{t}U_n(t)\right)\bigg|\\
	\leq \sqrt{n}\left|\frac{j}{k}-\frac{s}{t}\right|\left|U_n(t)\right|
	\leq \frac{\sqrt{n}}{k}|U_n(k/n)|\leq M\max_{1\leq i\leq n}|\eps_{i}|.
	\end{multline*}
	Dealing with the second summand in the denominators in the same way, we conclude that
	\begin{align*}
	&\sup_{t\in[0,1]}\Bigg|\max_{j\leq [nt]}\bigg|\sum_{i=1}^j\Big(\sigma(i/n)\eps_{i}-\frac{1}{[nt]}\sum_{\ell=1}^{[nt]}\sigma(\ell/n)\eps_{\ell}\Big)\bigg|\Bigg.\\
	&\quad\Bigg.+\max_{j> [nt]}\bigg|\sum_{i=j}^n\Big(\sigma(i/n)\eps_{i}-\frac{1}{n-[nt]}\sum_{\ell=[nt]+1}^{n}\sigma(\ell/n)\eps_{\ell}\Big)\bigg|-\sqrt{n}\sup_{s\in[0,t]}\left|U_n(s)-\frac{s}{t}U_n(t)\right|\Bigg.\\
	&\quad-\Bigg.\sqrt{n}\sup_{s\in[t,1]}\Big|U_n(1)-U_n(s)-\frac{1\!-\!s}{1\!-\!t}(U_n(1)-U_n(t))\Big|\Bigg|\leq 2M\max_{1\leq i\leq n}|\eps_{n,i}|=\op(b_n),\quad n\to\infty.
	\end{align*}
	Together with line~\eqref{assuprop1}, the statement of the proposition follows.
\end{proof}

\begin{proposition}\label{prop:discreteapprox1b}
	Assume that there is a~sequence $\{c_n\}_{n\in\N}$ such that
	\begin{multline}\label{assuprop1b}
	\prob\bigg[\min_{1\leq k\leq n}\bigg\{\sum_{i=1}^k\Big[\sum_{j=1}^{i}\Big(\sigma_n(j/n)\eps_{j}-\frac{1}{k}\sum_{\ell=1}^k\sigma(\ell/n)\eps_{\ell}\Big)\Big]^2\bigg.\bigg.\\
	\bigg.\bigg.+\sum_{i+k+1}^n\Big[\sum_{j=i}^{n}\Big(\sigma_n(j/n)\eps_{j}-\frac{1}{n-k}\sum_{\ell=k+1}^n\sigma(\ell/n)\eps_{\ell}\Big)\Big]^2\bigg\}\geq c_n\bigg]\rightarrow 1
	\end{multline}
	and 
	\begin{equation}\label{assuprop2b}
	n^2\max_{1\leq k\leq n}\eps_{k}^2=\op (c_n).
	\end{equation}
	Then, under the null hypothesis $\mathcal{H}_0$, $\mathscr{R}(V_n)-\mathscr{T}(U_n)\xrightarrow[n\to\infty]{\prob}0$.
\end{proposition}

\begin{proof}
	Recall that $\frac{1}{\sqrt{n}}\sum_{i=1}^k\Big(\sigma_n(i/n)\eps_{i}-\frac{1}{n}\sum_{j=1}^n\sigma(j/n)\eps_{j}\Big)=U_n(k/n)-\frac{k}{n}U_n(1)$. By Assumption~\ref{assump:UP2}, we get $\sup_{t\in[0,1]}\sigma(t)=:M<\infty$. For $s\in[0,t]$, let $j=[sn]$. Thus,
	\begin{align*}
	&\bigg|\bigg\{\sum_{i=1}^j\Big(\sigma_n(i/n)\eps_{i}-\frac{1}{k}\sum_{\ell=1}^k\sigma(\ell/n)\eps_{\ell}\Big)\bigg\}^2-n\left(U_n(s)-\frac{s}{t}U_n(t)\right)^2\bigg|\\
	&=n\left|\left[U_n(j/n)-j/kU_n(k/n)\right]^2-\left[U_n(s)-s/tU_n(t)\right]^2\right|\\
	&\leq n\left\{2\left|\frac{j}{k}-\frac{s}{t}\right|\left|U_n(s)\right|\left|U_n(t)\right|+\left|\frac{j}{k}-\frac{s}{t}\right|\left|\frac{j}{k}+\frac{s}{t}\right|U_n^2(t)\right\}\leq 4M^2k\max_{1\leq i\leq k}\eps_{i}^2\leq 4M^2n\max_{1\leq i\leq n}\eps_{i}^2.
	\end{align*}
	Dealing with the second summand in the denominators in the same way, we conclude that
	\begin{align*}
	&\sup_{t\in[0,1]}\Bigg|\sum_{j=1}^{[nt]}\bigg\{\sum_{i=1}^j\Big(\sigma(i/n)\eps_{i}-\frac{1}{[nt]}\sum_{\ell=1}^{[nt]}\sigma(\ell/n)\eps_{\ell}\Big)\bigg\}^2\Bigg.\\
	&\quad\Bigg.+\sum_{j=[nt]+1}^n\bigg\{\sum_{i=j}^n\Big(\sigma(i/n)\eps_{i}-\frac{1}{n-[nt]}\sum_{\ell=[nt]+1}^{n}\sigma(\ell/n)\eps_{\ell}\Big)\bigg\}^2-n\int_0^t\left\{U_n(s)-\frac{s}{t}U_n(t)\right\}^2\ud s\Bigg.\\
	&\quad-\Bigg.n\int_t^1\left\{U_n(1)-U_n(s)-\frac{1-s}{1-t}(U_n(1)-U_n(t))\right\}^2\ud s\Bigg|\leq 8M^2n^2\max_{1\leq i\leq n}\eps_{i}^2=\op(c_n),\quad n\to\infty.
	\end{align*}
	Together with line~\eqref{assuprop1b}, the statement of the proposition follows.
\end{proof}

\begin{lemma}\label{lemma:discreteapprox2}
	Under Assumptions~\ref{assump:UP1} and~\ref{assump:UP2}, Assumptions~\eqref{assuprop1}--\eqref{assuprop2b} hold.
\end{lemma}

\begin{proof}
	By Assumption~\ref{assump:UP2}, we have $\sup_{t\in[0,1]}\sigma(t)=:M<\infty$. We can choose a~sequence $\{b_n\}_{n\in\N}$ such that $b_n=o(\sqrt{n})$ and
	\[
	\prob\left[\max_{1\leq i\leq n}\frac{|\eps_{i}|}{b_n}\geq \xi\right]\leq \sum_{i=1}^n\prob\left(|\eps_{i}|\geq b_n\xi\right)\leq \frac{1}{b_n^p\xi^p}\sum_{i=1}^n\E|\eps_i|^p\xrightarrow{n\rightarrow\infty}0
	\]
	for any $\xi>0$ and some $p>2$ according to Assumption~\ref{assump:UP1}. Thus, relation~\eqref{assuprop2} holds. On the other hand, by the arguments from the proof of Proposition~\ref{prop:discreteapprox1}, one gets
	\begin{align*}
	&\min_{1\leq k\leq n}\Bigg(\max_{\ell\leq k}\left|\sum_{i=1}^{\ell}\left(\sigma_n(i/n)\eps_{i}-\frac{1}{k}\sum_{j=1}^k\sigma(j/n)\eps_j\right)\right|\Bigg.\\
	&\quad\Bigg.+\max_{\ell> k}\left|\sum_{i=\ell}^n\left(\sigma_n(i/n)\eps_{i}-\frac{1}{n-k}\sum_{j=k+1}^n\sigma(j/n)\eps_j\right)\right|\Bigg)\geq \sqrt{n}\inf_{t\in[0,1]}\Bigg(\sup_{s\in[0,t]}\Big|U_n(s)-\frac{s}{t}U_n(t)\Big|\Bigg.\\
	&\quad\Bigg.+\sup_{s\in[t,1]}\Big|U_n(1)-U_n(s)-\frac{1-s}{1-t}(U_n(1)-U_n(t))\Big|\Bigg)-2M\max_{1\leq i\leq n}|\eps_{i}|.
	\end{align*}
	We have already shown that $\max_{1\leq i\leq n}|\eps_{i}|=\op (b_n)$. Furthermore, we have the weak convergence
	\begin{multline*}
	\inf_{t\in[0,1]}\bigg(\sup_{s\in[0,t]}\Big|U_n(s)-\frac{s}{t}U_n(t)\Big|+\sup_{s\in[t,1]}\Big|U_n(1)-U_n(s)-\frac{1-s}{1-t}(U_n(1)-U_n(t))\Big|\bigg)\\
	\xrightarrow[n\to\infty]{\dist} \inf_{t\in[0,1]}\bigg(\sup_{s\in[0,t]}\Big|W_{\eta}(s)-\frac{s}{t}W_{\eta}(t)\Big|+\sup_{s\in[t,1]}\Big|W(1)-W_{\eta}(s)-\frac{1-s}{1-t}(W(1)-W_{\eta}(t))\Big|\bigg).
	\end{multline*}
	The fact that $b_n=o(\sqrt{n})$ easily implies~\eqref{assuprop1}. Moreover, we can choose a~sequence $\{c_n\}_{n\in\N}$ such that $c_n=o(n^3)$ and, for every $\xi>0$,
	\[
	\prob\left[n^2\max_{1\leq i\leq n}\frac{\eps_{i}^2}{c_n}\geq \xi\right]\leq \sum_{i=1}^n\prob\bigg(|\eps_{i}|\geq \frac{\sqrt{c_n\xi}}{n}\bigg)\leq \left(\frac{n^2}{c_n\xi}\right)^{p/2}\sum_{i=1}^n\E|\eps_i|^p\xrightarrow{n\rightarrow\infty}0
	\]
	due to Assumption~\ref{assump:UP1} for some $p>2$. Therefore, relation~\eqref{assuprop2b} holds. On the other hand, by the arguments from the proof of Proposition~\ref{prop:discreteapprox1b}, one gets
	\begin{align*}
	&\min_{1\leq k\leq n}\bigg(\sum_{l=1}^{k}\bigg\{\sum_{i=1}^l\Big(\sigma_n(i/n)\eps_{i}-\frac{1}{k}\sum_{j=1}^k\sigma(j/n)\eps_j\Big)\bigg\}^2\bigg.\\
	&\quad\bigg.+\sum_{l=k+1}^n\bigg\{\sum_{i=l}^n\Big(\sigma_n(i/n)\eps_{i}-\frac{1}{n-k}\sum_{j=k+1}^n\sigma(j/n)\eps_j\Big)\bigg\}^2\bigg)\geq n^3\inf_{t\in[0,1]}\bigg(\int_0^t\Big\{U_n(s)-\frac{s}{t}U_n(t)\Big\}^2\ud s\bigg.\\
	&\quad\bigg.+\int_t^1\Big\{U_n(1)-U_n(s)-\frac{1-s}{1-t}(U_n(1)-U_n(t))\Big\}^2\ud s\bigg)-8M^2n^2\max_{1\leq i\leq n}\eps_{i}^2.
	\end{align*}
	We have already shown that $n^2\max_{1\leq i\leq n}\eps_{n,i}^2=\op (c_n)$. Hence, we have the weak convergence
	\begin{multline*}
	\inf_{t\in[0,1]}\bigg(\int_0^t\Big\{U_n(s)-\frac{s}{t}U_n(t)\Big\}^2\ud s+\int_t^1\Big\{U_n(1)-U_n(s)-\frac{1-s}{1-t}(U_n(1)-U_n(t))\Big\}^2\ud s\bigg)\\
	\xrightarrow[n\to\infty]{\dist} \inf_{t\in[0,1]}\bigg(\int_0^t\Big\{W_{\eta}(s)-\frac{s}{t}W_{\eta}(t)\Big\}^2\ud s+\int_t^1\Big\{W(1)-W_{\eta}(s)-\frac{1-s}{1-t}(W(1)-W_{\eta}(t))\Big\}^2\ud s\bigg).
	\end{multline*}
	The fact that $c_n=o(n^3)$ easily implies~\eqref{assuprop1b}.
\end{proof}

\begin{proposition}\label{prop:conditionalvar}
	Under the Assumptions \ref{assump:UP1} and
	\ref{assump:UP2}, let $\delta_n\rightarrow 0$ as $n\rightarrow 0$. Then
	\begin{equation*}
	\sup_{t\in[0,1]}\bigg|\var\bigg[\frac{1}{\sqrt{n}}\sum_{k=1}^{[nt]}Y_{n,k}^\star\big|\ \{Y_{n,k}\}_{k\leq n}\bigg]-\nu(t)\bigg|\xrightarrow{n\rightarrow\infty}0
	\end{equation*}
	almost surely, where
	\begin{equation}\label{eq:nuoft}
	\nu(t)=\var\varepsilon_1\int_{0}^t\sigma^2(s)\ud s.
	\end{equation}
	If instead $\delta_n=\delta\neq 0$ and $\tau_n=[n\zeta]$ for some $\zeta\in(0,1)$, then
	\begin{equation*}
	\sup_{t\in[0,1]}\bigg|\var\bigg[\frac{1}{\sqrt{n}}\sum_{k=1}^{[nt]}Y_{n,k}^\star\big|\ \{Y_{n,k}\}_{k\leq n}\bigg]-\nu_{\delta}(t)\bigg|\xrightarrow{n\rightarrow\infty}0
	\end{equation*}
	almost surely, where
	\begin{equation*}
	\nu_{\delta}(t)=
	\begin{cases}
	\nu(t)+\delta^2(1-\zeta)^2t, &t\leq \zeta;\\
	\nu(t)+\delta^2\zeta(1-2\zeta+\zeta t), &t> \zeta.
	\end{cases}
	\end{equation*}
\end{proposition}

\begin{proof} Without loss of generality, we may assume that $\mu=0$ in our Model~\eqref{eq:locmodel}. Recall that $Y_{n,k}^{\star}:=\left(Y_{n,k}-\bar{Y}_{n,1:n}\right)X_{k}$. Because our multipliers $\{X_n\}_{n\in\N}$ satisfy $\E X_n=0$, $\var X_n=1$ and are uncorrelated, we have
	\begin{align*}
	&\var\bigg[\frac{1}{\sqrt{n}}\sum_{k=1}^{[nt]}Y_{n,k}^\star\big|\ \{Y_{n,k}\}_{k\leq n}\bigg]=\frac{1}{n}\sum_{k=1}^{[nt]}\left(Y_{n,k}-\bar{Y}_{n,1:n}\right)^2\var X_k=\frac{1}{n}\sum_{k=1}^{[nt]}\left(Y_{n,k}-\bar{Y}_{n,1:n}\right)^2\\
	&=\frac{1}{n}\sum_{k=1}^{[nt]}Y^2_{n,k}+\bigg\{\frac{[nt]}{n}\Big(\frac{1}{n}\sum_{k=1}^{n}Y_{n,k}\Big)^2-\frac{2}{n^2}\sum_{k=1}^{n}Y_{n,k}\sum_{k=1}^{[nt]}Y_{n,k}\bigg\}=:I_n(t)+I\!I_n(t).
	\end{align*}
	We will treat the two summands separately. By our Assumption~\ref{assump:UP2}, we have $M:=\sup_{t\in[0,1]}\sigma(t)<\infty$, so $\E|\sigma(k/n)\varepsilon_k|^p\leq C_p$ for some $C_p<\infty$ uniformly in~$k$ and~$n$, because of our Assumption~\ref{assump:UP1}. Furthermore, the $\alpha$-mixing coefficients of the triangular scheme  $\sigma(k/n)\varepsilon_k$, $k\leq n$, $n\in\N$ are the same as the mixing coefficients of the sequence $\varepsilon_n$, $n\in\N$.
	
	For the first summand $I_n(t)$, we obtain
	\[
	\E I_n(t)=\frac{1}{n}\sum_{k=1}^{[nt]}\E Y^2_{n,k}=\frac{1}{n}\sum_{k=1}^{[nt]}\var Y_{n,k}+\frac{1}{n}\sum_{k=1}^{[nt]}\E^2Y_{n,k}=\frac{1}{n}\sum_{k=1}^{[nt]}\sigma^2\left(\frac{i}{n}\right)\var\varepsilon_1+\frac{1}{n}\sum_{k=1}^{[nt]}\E^2Y_{n,k}.
	\]
	Note that $|\E Y_{n,k}|\leq|\delta_n|$. We will firstly treat the case $\delta_n^2\xrightarrow{n\to\infty} 0$. So, $\sup_{t\in[0,1]}\left|\E I_n(t) -\nu(t)\right|\xrightarrow{n\rightarrow\infty}0$. We will proceed by showing that $I_n(t)$ converges uniformly to its expectation almost surely. For this, we define
	\begin{align*}
	Z_{n,k,B}&:=Y^2_{n,k}\ind\{|Y^2_{n,k}|\leq B\}-\E Y^2_{n,k}\ind\{|Y^2_{n,k}|\leq B\};\\
	\tilde{Z}_{n,k,B}&:=Y^2_{n,k}-\E Y^2_{n,k}]-Z_{n,k}=Y^2_{n,k}\ind\{|Y^2_{n,k}|> B\}-\E Y^2_{n,k}\ind\{|Y^2_{n,k}|> B\}.
	\end{align*}
	We will chose $B\equiv B_n=n^{1/8}$. By definition $\E Z_{n,k,B}=\E \tilde{Z}_{n,k,B}=0$ and $Z_{n,k,B}$ is bounded by~$B_n$. We can now apply Theorem~5 of~\cite{kim1994moment} (in a~simplified version for bounded random variables similar to Theorem~2 of~\cite{Yokoyama1980}) to conclude that $\E\left(\sum_{k=m_1}^{m_2}Z_{n,k,B}\right)^4\leq C(m_2-m_1)^2B^4_n$ for all $m_1<m_2\leq n$. By Theorem~1 of~\cite{moricz1976moment}, it follows that $\E\left(\max_{m\leq n}\bigg|\sum_{k=1}^{m}Z_{n,k,B}\bigg|\right)^4\leq Cn^2B^4_n=Cn^{5/2}$. With the Chebyshev inequality, we obtain for every $\epsilon>0$,
	\[
	\sum_{n=1}^\infty \P\bigg[\max_{m\leq n}\bigg|\frac{1}{n}\sum_{k=1}^{m}Z_{n,k,B}\bigg|>\epsilon\bigg]\leq \sum_{i=1}^\infty \frac{C}{\epsilon^2n^{3/2}}<\infty
	\]
	and the Borel–Cantelli lemma implies that $\max_{m\leq n}|\frac{1}{n}\sum_{k=1}^{m}Z_{n,k,B}|\rightarrow 0$ as $n\rightarrow 0$ almost surely. To treat the values $\tilde{Z}_{n,k,B}$, note that $Y^2_{n,k}\ind\{|Y^2_{n,k}|> B\}\leq Y^2_{n,k}\ind\{|Y^2_{n,k}|> B'\}$ for $B'\leq B$ and, consequently, $|\tilde{Z}_{n,k,B}|\leq |\tilde{Z}_{n,k,B'}|+\E|Y^2_{n,k}\ind\{|Y^2_{n,k}|> B\}|+\E|Y^2_{n,k}\ind\{|Y^2_{n,k}|> B'\}|$. Furthermore, $\E|Y_{k,n}|^p\leq C_p$ uniformly in~$k$ and~$n$. Thus, $\E|\tilde{Z}_{n,k,B}|\leq KB^{-(p-2)/2}$ for some constant $K<\infty$. For $B'\leq B$, we get $|\tilde{Z}_{n,k,B}|\leq |\tilde{Z}_{n,k,B'}|+\frac{2K}{B'^{(p-2)/2}}$. Hence, we obtain
	\begin{align*}
	&\E\bigg(\max_{n=2^l+1,\ldots,2^l}\max_{m\leq n}\Big|\frac{1}{n}\sum_{k=1}^m\tilde{Z}_{n,k,B_n}\Big|\bigg)\leq \E\bigg(\max_{n=2^l+1,\ldots,2^{l+1}}\max_{m\leq n}\frac{1}{2^l}\sum_{k=1}^m\left(|\tilde{Z}_{n,k,B_{2^l}}|+2KB_{2^l}^{-(p-2)/2}\right)\bigg)\\
	&\leq \E\bigg(\frac{1}{2^l}\sum_{k=1}^{2^{l+1}}\left(|\tilde{Z}_{n,k,B_{2^l}}|+2KB_{2^l}^{-(p-2)/2}\right)\bigg)\leq \frac{1}{2^l}\sum_{k=1}^{2^{l+1}}3KB_{2^l}^{-(p-2)/2}=6K2^{-ql}
	\end{align*}
	with $q=(p-2)/16>0$. By the Markov inequality for every $\epsilon>0$,
	\begin{equation*}
	\sum_{l=1}^\infty \P\left[\max_{n=2^l+1,\ldots,2^l}\max_{m\leq n}\bigg|\frac{1}{n}\sum_{k=1}^m\tilde{Z}_{n,k,B_n}\bigg|>\epsilon\right]\leq \sum_{i=1}^\infty \frac{6}{\epsilon}K2^{-ql}<\infty
	\end{equation*}
	and using the Borel–Cantelli lemma again, we have shown that $\max_{m\leq n}\frac{1}{n}\bigg|\sum_{k=1}^{m}\tilde{Z}_{n,k,B}\bigg|\rightarrow 0$ and as $n\rightarrow 0$ almost surely. So, we arrive at $\sup_{t\in[0,1]}\left|I_n(t)-\E I_n(t)\right|\leq \max_{m\leq n}\bigg|\frac{1}{n}\sum_{k=1}^{m}Z_{n,k,B}\bigg|+\max_{m\leq n}\frac{1}{n}\bigg|\sum_{k=1}^{m}\tilde{Z}_{n,k,B}\bigg|\xrightarrow{n\rightarrow\infty}0$ almost surely.
	
	For the second summand $I\!I_n(t)$, note that $\bigg|\E\bigg(\frac{1}{n}\sum_{k=1}^{[nt]}Y_{n,k}\bigg)\bigg|\leq |\delta_n|\rightarrow 0$ uniformly in $t\in[0,1]$. It remains to show that $\sup_{t\in[0,1]}\bigg|\frac{1}{n}\sum_{k=1}^{[nt]}Y_{n,k}-\E\bigg(\frac{1}{n}\sum_{k=1}^{[nt]}Y_{n,k}\bigg)\bigg|\rightarrow 0$ almost surely as $n\rightarrow\infty$. This can be proved along the lines of the convergence of~$I_n(t)$ and is, hence, omitted.
	
	For fixed alternatives ($\delta_n\equiv\delta\neq 0$ and $\tau_n=[n\zeta]$), note that $\E Y_{n,k}=\delta\ind\{k> \tau_n\}$. Furthermore, $\var(\frac{1}{n}\sum_{k=1}^mY_{k,n})\xrightarrow{n\rightarrow\infty}0$. Thus, $\E\Big(\frac{1}{n}\sum_{k=1}^{m_1}Y_{k,n}\frac{1}{n}\sum_{k=1}^{m_2}Y_{k,n}\Big)-\E\Big(\frac{1}{n}\sum_{k=1}^{m_1}Y_{k,n}\Big)\E\Big(\frac{1}{n}\sum_{k=1}^{m_2}Y_{k,n}\Big)\xrightarrow{n\rightarrow\infty}0$. We conclude that $\E I_n(t)\approx \nu(t)+\delta^2(t-\zeta)^+$ and $\E I\!I_n(t)\approx \delta^2\big(t(1-\zeta)^2+ -2(1-\zeta)(t-\zeta)^+\big)$, where the convergence holds uniformly in $t\in[0,1]$. Following the arguments for the case $\delta_n\rightarrow 0$, it can be shown that $\sup_{t\in[0,1]}\left|\var\bigg[\frac{1}{\sqrt{n}}\sum_{k=1}^{[nt]}Y_{n,k}^\star\big|\ \{Y_{n,k}\}_{k\leq n}\bigg]-\E I_n-\E I\!I_n\right|\rightarrow 0$ almost surely as $n\rightarrow\infty$. Finally, simple algebra gives $\E I_n(t)+\E I\!I_n(t)\rightarrow \nu_{\delta}(t)$.
\end{proof}

\subsection{Proofs of the main results}

\begin{proof}[Proof of Theorem~\ref{theorem:undernull}]
	A~consequence of Corollary~\ref{corollary:limitdist}, Propositions~\ref{prop:discreteapprox1}, \ref{prop:discreteapprox1b}, and Lemma~\ref{lemma:discreteapprox2}.
\end{proof}

\begin{proof}[Proof of Theorem~\ref{theorem:underalt}]
	With respect to Assumptions~\ref{assump:UP1}, \ref{assump:UP2} and according to the underlying proof of Theorem~\ref{theorem:undernull}, we have, as $n\to\infty$,
	\begin{align}
	\frac{1}{\sqrt{n}}\left|\sum_{i=1}^{\tau_n}\left(\sigma(i/n)\eps_{i}-\frac{1}{n}\sum_{j=1}^n\sigma(j/n)\eps_j\right)\right|&=\Op(1);\label{eq:Op1a}\\
	F_1(n):=\max\limits_{1\leq i \leq \tau_n}\frac{1}{\sqrt{n}}\left|\sum_{j=1}^{i}\left(\sigma(j/n)\eps_{j}-\frac{1}{\tau_n}\sum_{\ell=1}^{\tau_n}\sigma(\ell/n)\eps_{\ell}\right)\right|&=\Op(1);\label{eq:Op1b}\\
	F_2(n):=\max\limits_{\tau_n< i \leq n}\frac{1}{\sqrt{n}}\left|\sum_{j=i}^{n}\left(\sigma(j/n)\eps_{j}-\frac{1}{n-\tau_n}\sum_{\ell=\tau_n+1}^{n}\sigma(\ell/n)\eps_{\ell}\right)\right|&=\Op(1).\label{eq:Op1c}
	\end{align}
	
	Note that there are no changes in the expectation of $Y_{n,1},\ldots,Y_{n,\tau_n}$ as well as in the expectation of $Y_{n,\tau_n+1},\ldots,Y_{n,n}$. Let $k=\tau_n$. Then, under~$\mathcal{H}_1$ and due to~\eqref{eq:Op1a}--\eqref{eq:Op1c} and Assumption~\ref{assump:UP3},
	\begin{align*}
	\mathscr{Q}(V_n)&\geq \frac{\left|\sum_{i=1}^{\tau_n}\left(Y_{n,i}-\bar{Y}_{n,1:n}\right)\right|}{\max\limits_{1\leq i \leq \tau_n}\left|\sum_{j=1}^{i}\left(Y_{n,j}-\bar{Y}_{n,1:\tau_n}\right)\right|+\max\limits_{\tau_n< i \leq n}\left|\sum_{j=i}^{n}\left(Y_{n,j}-\bar{Y}_{n,(\tau_n+1):n}\right)\right|}\\
	&=\frac{n^{-1/2}\left|\sum_{i=1}^{\tau_n}\left(\sigma(i/n)\eps_{i}-\frac{1}{n}\sum_{j=1}^n\sigma(j/n)\eps_j-\d_n(n-\tau_n)/n\right)\right|}{F_1(n)+F_2(n)}\\
	&=\Op(1)+\frac{|\delta_n|\sqrt{n}\left(1-\frac{\tau_n}{n}\right)\frac{\tau_n}{n}}{F_1(n)+F_2(n)}\xrightarrow[n\to\infty]{\prob}\infty.
	\end{align*}
	
	Similarly for the second self-normalized test statistic,
	\begin{align}
	\frac{1}{n}\sum_{k=1}^{n}\left\{\frac{1}{\sqrt{n}}\sum_{i=1}^{k}\left(\sigma(i/n)\eps_{i}-\frac{1}{n}\sum_{j=1}^n\sigma(j/n)\eps_j\right)\right\}^2&=\Op(1);\label{eq:Op2a}\\
	G_1(n):=\frac{1}{n}\sum_{i=1}^{\tau_n}\left\{\frac{1}{\sqrt{n}}\sum_{j=1}^{i}\left(\sigma(j/n)\eps_{j}-\frac{1}{\tau_n}\sum_{\ell=1}^{\tau_n}\sigma(\ell/n)\eps_{\ell}\right)\right\}^2&=\Op(1);\label{eq:Op2b}\\
	G_2(n):=\frac{1}{n}\sum_{i=\tau_n+1}^n\left\{\frac{1}{\sqrt{n}}\sum_{j=i}^{n}\left(\sigma(j/n)\eps_{j}-\frac{1}{n-\tau_n}\sum_{\ell=\tau_n+1}^{n}\sigma(\ell/n)\eps_{\ell}\right)\right\}^2&=\Op(1).\label{eq:Op2c}
	\end{align}
	Under the alternative hypothesis~$\mathcal{H}_1$ together with realizing~\eqref{eq:Op2a}--\eqref{eq:Op2c} and Assumption~\ref{assump:UP3}, we obtain
	\begin{align*}
	&\plim_{n\to\infty}\mathscr{S}(V_n)\geq \plim_{n\to\infty}\frac{\sum_{k=1}^{n}\left\{\sum_{i=1}^{k}\left(Y_{n,i}-\bar{Y}_{n,1:n}\right)\right\}^2}{\sum_{i=1}^{\tau_n}\left\{\sum_{j=1}^{i}\left(Y_{n,j}-\bar{Y}_{n,1:\tau_n}\right)\right\}^2+\sum_{i=\tau_n+1}^n\left\{\sum_{j=i}^{n}\left(Y_{n,j}-\bar{Y}_{n,(\tau_n+1):n}\right)\right\}^2}\\
	&=\plim_{n\to\infty}\frac{\sum_{k=1}^{n}\left\{\sum_{i=1}^{k}\left(\sigma(i/n)\eps_{i}-\frac{1}{n}\sum_{j=1}^n\sigma(j/n)\eps_j\right)/\sqrt{n}-\d_n\sqrt{n}f_n(k,\tau_n)\right\}^2/n}{G_1(n)+G_2(n)}=\infty,
	\end{align*}
	again because there are no changes in the means of $Y_{n,1},\ldots,Y_{n,\tau_n}$ as well as in the means of $Y_{n,\tau_n+1},\ldots,Y_{n,n}$, where
	\[
	f_n(k,\tau_n):=\left\{
	\begin{array}{ll}
	k\left(1-\tau_n/n\right), & k\leq \tau_n;\\
	\tau_n\left(1-k/n\right), & k> \tau_n.
	\end{array}
	\right.
	\]
\end{proof}

\begin{proof}[Proof of Theorem~\ref{theorem:boot}] Since $Y_{n,k}^{\star}:=\left(Y_{n,k}-\bar{Y}_{n,1:n}\right)X_{k}$, $k=1,\ldots,n$ for i.i.d.~standard normal random variables $\{X_{n}\}_{n=1}^{\infty}$, the bootstrap partial sum process $S_n^\star(t)$ with
	\begin{equation*}
	S_n^\star(t):=\frac{1}{\sqrt{n}}V_n^\star([nt])=\frac{1}{\sqrt{n}}\sum_{k=1}^{[nt]}Y_{n,k}^{\star}
	\end{equation*}
	for $t\in[0,1]$ has conditionally on $\{Y_{n,k}\}_{k=1}^{n}$ the same distribution as $\left(W\left(\nu_n^\star(t)\right)\right)_{t\in[0,1]}$ for some standard Wiener process~$W$ and
	\begin{equation*}
	\nu_n^\star(t)=\var\left[\frac{1}{\sqrt{n}}\sum_{k=1}^{[nt]}Y_{n,k}^{\star}\big|\ \{Y_{n,k}\}_{k\leq n}\right].
	\end{equation*}
	If $\delta_n\rightarrow 0$ as $n\rightarrow \infty$, we have that $\nu_n^\star\rightarrow \nu$ uniformly in $t\in[0,1]$ almost surely, where $\nu$ is defined in~\eqref{eq:nuoft}. By Proposition~\ref{prop:conditionalvar} and the uniform continuity of~$W$, we know that, conditionally on $\{Y_{n,k}\}_{k=1}^{n}$, $\sup_{t\in[0,1]}\left|W\left(\nu_n^\star(t)\right)-W\left(\nu(t)\right)\right|\xrightarrow{n\rightarrow\infty}0$ almost surely. We conclude that
	\begin{equation*}
	\big(S_n^\star(t)\big)_{t\in[0,1]}\xrightarrow[n\to\infty]{\weak}\big(W(\nu(t))\big)_{t\in[0,1]}
	\end{equation*}
	and it follows by Lemma~\ref{lem:continuity} that
	\begin{equation*}
	\mathscr{S}(S_n^{\star})\xrightarrow[n\to\infty]{\dist}\mathscr{S}\left(\big(W(\nu(t))\big)_{t\in[0,1]}\right)\quad\mbox{and}\quad\mathscr{T}(S_n^{\star})\xrightarrow[n\to\infty]{\dist}\mathscr{T}\left(\big(W(\nu(t))\big)_{t\in[0,1]}\right).
	\end{equation*}
	Following the arguments of Propositions \ref{prop:discreteapprox1} and \ref{prop:discreteapprox1b}, we can show that $\mathscr{S}(S_n^{\star})-\mathscr{Q}(V_n^\star)\xrightarrow[n\to\infty]{\prob}0$ and $\mathscr{T}(S_n^{\star})-\mathscr{R}(V_n^\star)\xrightarrow[n\to\infty]{\prob}0$. For this, we use the fact that for Gaussian random variables $X_k$, $k\leq n$, we have $\max_{k\leq n}|X_k|=O(\log(n))$ almost surely, so $\max_{k\leq n}|Y_{n,k}^\star|=O\left(\log(n)\max_{k\leq n}|\epsilon_k|\right)$. Now, $\nu(t)=c\eta(t)$ for some constant~$c$. Then, $(W(\nu(t)))_{t\in[0,1]}$ and $c^{1/2}W_{\eta}$ have the same distribution. Furthermore, the functionals~$\mathscr{S}$ and~$\mathscr{T}$ are invariant under scale changes and, thus, $\mathscr{Q}(V_n^\star)\xrightarrow[n\to\infty]{\dist}\mathscr{S}\left(W_{\eta}\right)$, $\mathscr{R}(V_n^\star)\xrightarrow[n\to\infty]{\dist}\mathscr{T}\left(W_{\eta}\right)$.
	
	Under the fixed alternative ($\delta_n\equiv\delta\neq 0$), the same arguments together with the second part of Proposition~\ref{prop:conditionalvar} lead to
	\begin{equation*}
	\mathscr{Q}(V_n^\star)\xrightarrow[n\to\infty]{\dist}\mathscr{S}\left(\big(W(\nu_\delta(t))\big)_{t\in[0,1]}\right)\quad\mbox{and}\quad
	\mathscr{R}(V_n^\star)\xrightarrow[n\to\infty]{\dist}\mathscr{T}\left(\big(W(\nu_\delta(t))\big)_{t\in[0,1]}\right).
	\end{equation*}
	A~simple algebra reveals that the two centered Gaussian processes $W_{\eta}-\frac{\delta}{\varsigma}B_{\zeta}$ and $(W(\nu_\delta(t)))_{t\in[0,1]}$ have the same covariance structure (up to a~constant), which completes the proof.
\end{proof}

\begin{proof}[Proof of Theorem~\ref{theorem:estimation}]
	We can rewrite the estimator as
	\begin{equation*}
	\hat{\tau}_n=\operatorname{argmax}\frac{\frac{1}{n|\delta_n|}\big|V_n(k)-k/n V_n(n)\big|+\frac{1}{n|\delta_n|}\big|\widetilde{V}_n(n-k)-k/n V_n(n)\big|}{\max\limits_{1\leq i \leq k}\frac{1}{\sqrt{n}}\big|V_n(i)-i/kV_n(k)\big|+\max\limits_{k< i \leq n}\frac{1}{\sqrt{n}}\big|\widetilde{V}_n(i)-(n-i)/(n-k)\widetilde{V}_n(k)\big|}.
	\end{equation*}
	We will treat the numerator $E_n(k)$ and the denominator $D_n(k)$ of the above stated ratio separately. The numerator can be decomposed as
	\begin{align}
	E_n(k)&=\frac{1}{n|\delta_n|}\big|V_n(k)-k/n V_n(n)\big|+\frac{1}{n|\delta_n|}\big|\widetilde{V}_n(n-k)-k/n V_n(n)\big|\nonumber\\
	&=\bigg|\frac{1}{n|\delta_n|}\bigg(\sum_{i=1}^k\E Y_{n,i}-\frac{k}{n} \sum_{i=1}^n\E Y_{n,i}\bigg)+\frac{1}{n|\delta_n|}\bigg(\sum_{i=1}^k\epsilon_{n,i}-\frac{k}{n}\sum_{i=1}^n\epsilon_{n,i}\bigg)\bigg|\label{eq:Enk1}\\
	&\quad+\bigg|\frac{1}{n|\delta_n|}\bigg(\!\sum_{i=n-k+1}^n\E Y_{n,i}-\frac{k}{n} \sum_{i=1}^n\E Y_{n,i}\bigg)+\frac{1}{n|\delta_n|}\bigg(\!\sum_{i=n-k+1}^n\epsilon_{n,i}-\frac{k}{n}\sum_{i=1}^n\epsilon_{n,i}\bigg)\bigg|\label{eq:Enk2}.
	\end{align}
	The second summands in~\eqref{eq:Enk1} and~\eqref{eq:Enk2} converge to~$0$ uniformly using the functional central limit theorem by~\citet[Lemma~1 and~2]{Cav2005} and Assumption~\ref{assump:UP3}. For the first summands, recall that by our model $\E Y_{i,n}=\mu+\delta_n\ind\{i>\tau_n\}$. It is easy to see that
	\begin{equation*}
	\frac{\left|\sum\limits_{i=1}^{[nt]}\E Y_{n,i}-\frac{[nt]}{n} \sum\limits_{i=1}^n\E Y_{n,i}\right|+\left|\sum\limits_{i=[nt]+1}^{n}\E Y_{n,i}-\frac{[nt]}{n}\sum\limits_{i=1}^n\E Y_{n,i}\right|}{n|\delta_n|}
	\xrightarrow{n\to\infty}
	\begin{cases}
	t, & t\leq\zeta\leq 1-t;\\
	\zeta, & \zeta\leq \min\{t,1-t\};\\
	1-\zeta, & \zeta\geq \max\{t,1-t\};\\
	1-t, & 1-t\leq\zeta\leq t.
	\end{cases}
	\end{equation*}
	uniformly in~$t$. For $k=\tau_n$, we have that the denominator
	\begin{align*}
	D_n(\tau_n)&=\max\limits_{1\leq i \leq k}\frac{1}{\sqrt{n}}\bigg|V_n(i)-i/kV_n(k)\bigg|+\max\limits_{k< i \leq n}\frac{1}{\sqrt{n}}\bigg|\widetilde{V}_n(i)-(n-i)/(n-k)\widetilde{V}_n(k)\bigg|\\
	&\xrightarrow{\dist}\sup_{t\leq \zeta}\bigg|Z_{\eta,\zeta,\delta,\varsigma}(t)-\frac{t}{\zeta}Z_{\eta,\zeta,\delta,\varsigma}(\zeta)\bigg|+\sup_{t>\zeta}\bigg|\widetilde{Z}_{\eta,\zeta,\delta,\varsigma}(t)-\frac{1-t}{1-\zeta}(\widetilde{Z}_{\eta,\zeta,\delta,\varsigma}(\zeta))\bigg|=:Z,
	\end{align*}
	where
	\[
	Z_{\eta,\zeta,\delta,\varsigma}(t):=W_{\eta}(t)-\frac{\delta}{\varsigma}B_{\zeta}(t),\quad\widetilde{Z}_{\eta,\zeta,\delta,\varsigma}(t):=Z_{\eta,\zeta,\delta,\varsigma}(1)-Z_{\eta,\zeta,\delta,\varsigma}(t),
	\]
	and the random variable~$Z$ is strictly positive almost surely. We conclude that $|E_n(\tau_n)/D_n(\tau_n)|$ converge in distribution to the random variable $\zeta(1-\zeta)/Z$. For $k=[nt]$ with $t>\zeta$, we have that
	\begin{align*}
	&\max\limits_{1\leq i \leq [nt]}\frac{1}{\sqrt{n}}\bigg|V_n(i)-\frac{i}{[nt]}V_n([nt])\bigg|+\max\limits_{[nt]< i \leq n}\frac{1}{\sqrt{n}}\bigg|\widetilde{V}_n(i)-\frac{n-i}{n-[nt]}\widetilde{V}_n([nt])\bigg|\\
	&\geq \frac{1}{\sqrt{n}}\left|\sum_{i=1}^{\tau_n}\E Y_{n,i}-\tau_n/[nt] \sum_{i=1}^{[nt]} \E Y_{n,i}\right|\\
	&\quad-\max\limits_{k}\left(\max\limits_{1\leq i \leq k}\frac{1}{\sqrt{n}}\bigg|U_n\Big(\frac{i}{n}\Big)-\frac{i}{k}U_n\Big(\frac{k}{n}\Big)\bigg|+\max\limits_{k< i \leq n}\frac{1}{\sqrt{n}}\bigg|\widetilde{U}_n\Big(\frac{i}{n}\Big)-\frac{n-i}{n-k}\widetilde{U}_n\Big(\frac{k}{n}\Big)\bigg|\right).
	\end{align*}
	The last summand is stochastically bounded by~\citet[Lemma~1 and~2]{Cav2005}. For the first summand, we have
	\begin{equation*}
	\frac{1}{\sqrt{n}}\left|\sum_{i=1}^{\tau_n}\E Y_{n,i}-\tau_n/[nt] \sum_{i=1}^{[nt]}\E Y_{n,i}\right|=\frac{1}{\sqrt{n}}\left([nt]-[n\zeta]\right)\frac{[n\zeta]}{[nt]}|\delta_n|\approx \frac{(t-\zeta)\zeta}{t}|\delta_n|\sqrt{n}
	\xrightarrow{n\rightarrow\infty}\infty,
	\end{equation*}
	because of Assumption~\ref{assump:UP3}. Similar arguments can be applied in the case~$t<\zeta$ and the convergence holds uniformly for all~$t$ outside any $\epsilon$-neighborhood of~$\zeta$. It follows that for an arbitrary $\epsilon>0$, $\max_{k: |k-\tau_n|\geq n\epsilon}\frac{|E_n(k)|}{D_n(k)}=\Op\left(\frac{1}{|\delta_n|\sqrt{n}}\right)$. Now, let us chose a~sequence $d_n\rightarrow 0$ with $d_n|\delta_n|\sqrt{n}\rightarrow\infty$. Then, for any $\epsilon>0$,
	\begin{equation*}
	\P\left(|\hat{\tau}_n/n-\zeta|>\epsilon\right)\leq \P\left(|E_n(\tau_n)/D_n(\tau_n)|<d_n\right)+\P\left[\max_{k: |k-\tau_n|\geq n\epsilon}|E_n(k)/D_n(k)|>d_n\right]
	\xrightarrow{n\rightarrow\infty}0.
	\end{equation*}
\end{proof}

\end{document}